\documentclass[12pt]{amsart}

\usepackage{amsmath,amsthm,amsfonts,amssymb,eucal}

\renewcommand {\a}{ \alpha }
\renewcommand{\b}{\beta}

\newcommand{\g}{\gamma}
\newcommand{\G}{\Gamma}

\newcommand{\vark}{\varkappa}
\renewcommand{\d}{\delta}
\newcommand{\s}{\sigma}

\renewcommand{\L}{\Lambda}
\newcommand{\z}{\zeta}
\renewcommand{\t}{\theta}

\newcommand{\p}{\partial}
\newcommand{\om}{\omega}
\newcommand{\Om}{\Omega}

\newcommand{\oq}{\ {\raise 7pt\hbox{${\scriptstyle\circ}$}}
	\kern -7pt{
		\hbox{$Q$}}}

\newcommand{\R}{ \mathbb R}

\newcommand {\GH}{\mathfrak H}

\newcommand {\GS}{\mathfrak S}

\newcommand {\GV}{\mathfrak V}
\newcommand {\GW}{\mathfrak W}

\newcommand {\bx}{\mathbf x}

\newcommand {\be}{\mathbf e}

\newcommand {\bz}{\mathbf z}
\newcommand {\by}{\mathbf y}

\newcommand {\bn}{\mathbf n}

\newcommand{\SN}{{\sf{N}}}
\newcommand{\SX}{{\sf{X}}}

\newcommand {\bmu}{\boldsymbol\mu}

\newcommand {\boldeta}{\boldsymbol\eta}

\newcommand {\bxi}{\boldsymbol\xi}

\newcommand{\overc}{\overset{\circ}}

\newcommand{\lu}{\langle}
\newcommand{\ru}{\rangle}

\newcommand{\CK}{\mathcal K}
\newcommand{\CB}{\mathcal B}

\newcommand{\CH}{\mathcal H}
\newcommand{\CO}{\mathcal O}

\newcommand{\CA}{\mathcal A}

\newcommand{\CC}{\mathcal C}
\newcommand{\CS}{\mathcal S}

\newcommand{\CD}{\mathcal D}

\newcommand{\plainW}[2]{\textup{{\textsf{W}}}^{#1, #2}}
\newcommand{\plainC}[1]{\textup{{\textsf{C}}}^{#1}}

\newcommand{\plainS}{\textup{{\textsf{S}}}}

\newcommand{\plainL}[1]{\textup{{\textsf{L}}}^{#1}}

\DeclareMathOperator{\tr}{{tr}}

\newcommand{\1}
{{\,\vrule depth3pt height9pt}{\vrule depth3pt height9pt}
	{\vrule depth3pt height9pt}{\vrule depth3pt height9pt}\,}
\newcommand{\2}{\1\!\1}

\DeclareMathOperator {\im }{{Im}}

\DeclareMathOperator {\dist} {{dist}}

\DeclareMathOperator{\op}{{Op}}

\DeclareMathOperator{\supp}{{supp}}

\newtheorem{thm}{Theorem}[section]
\newtheorem{cor}[thm]{Corollary}
\newtheorem{lem}[thm]{Lemma}
\newtheorem{prop}[thm]{Proposition}
\newtheorem{cond}[thm]{Condition}

\theoremstyle{definition}

\newtheorem{rem}[thm]{Remark}

\numberwithin{equation}{section}

\newcommand{\bee}{\begin{equation}}
	\newcommand{\ene}{\end{equation}}
\newcommand{\bees}{\begin{equation*}}
	\newcommand{\enes}{\end{equation*}}
\newcommand{\bes}{\begin{split}}
	\newcommand{\ens}{\end{split}}

\newcommand{\bet}{\begin{thm}}
	\newcommand{\ent}{\end{thm}}
\newcommand{\bel}{\begin{lem}}
	\newcommand{\enl}{\end{lem}}
\newcommand{\bec}{\begin{cor}}
	\newcommand{\enc}{\end{cor}}
\newcommand{\bep}{\begin{proof}}
	\newcommand{\enp}{\end{proof}}
\newcommand{\ber}{\begin{rem}}
	\newcommand{\enr}{\end{rem}}

\newcommand{\Z}{\mathbb Z}

\usepackage{color}

\begin{document}
	\hoffset -4pc

\title
[Quasi-classical asymptotics]
{{Quasi-classical asymptotics for functions of 
Wiener-Hopf operators: smooth vs non-smooth symbols}}
\author[A.V. Sobolev]{Alexander V. Sobolev}
\address{Department of Mathematics\\ University College London\\
	Gower Street\\ London\\ WC1E 6BT UK}
\email{a.sobolev@ucl.ac.uk}
\keywords{Non-smooth functions of 
Wiener--Hopf operators, asymptotic 
trace formulas, entanglement entropy}
\subjclass[2010]{Primary 47G30, 35S05; 
Secondary 45M05, 47B10, 47B35}

\begin{abstract}	
We consider functions of  
Wiener--Hopf type operators on the Hilbert space $\plainL2(\mathbb R^d)$. 
It has been known for a long time that the quasi-classical asymptotics 
for traces of resulting operators strongly depend on the smoothness 
of the symbol: for smooth symbols the expansion is power-like, whereas 
discontinuous symbols (e.g. indicator functions) 
produce an extra logarithmic factor. 
We investigate the transition regime by studying symbols depending on 
an extra parameter $T\ge 0$ in such a way that the symbol tends to 
a discontinuous one as $T\to 0$. The main result is two-parameter 
asymptotics (in the quasi-classical parameter and in $T$),
  describing a transition from the smooth case 
to the discontinuous one. 
The obtained asymptotic formulas are used to analyse the low-temperature 
scaling limit of the spatially 
bipartite entanglement entropy of thermal equilibrium states of non-interacting 
fermions. 
\end{abstract}

\maketitle

\tableofcontents
\section{Introduction}   

The present paper is devoted to the study of (bounded, self-adjoint) operators 
of the form
\begin{align}\label{WH:eq}
W_\a := W_\a(a; \L) := \chi_\L \op_\a(a) \chi_\L,\ \a >0,
\end{align}
on $\plainL2(\R^d)$, $d\ge 1$, 
where $\chi_\L$ is the indicator function of a set $\L\subset\R^d$. 
The notation $\op_\a(a)$ stands for the $\a$-pseudo-differential 
operator with symbol $a = a(\bxi)$, which acts on Schwartz functions $u$ on $\R^d$ as
\begin{equation*}
\bigl(\op_\a(a) u\bigr)(\bx) := \frac{\a^{d}}{(2\pi)^d}
\iint e^{i\a\bxi\cdot(\bx-\by)} a(\bxi) u(\by) d\by d\bxi\,,\quad \bx\in\R^d.
\end{equation*} 
Integrals without indication of the 
integration domain always mean integration 
over $\R^d$ with the value of $d$ which is clear from the context.
%
%
We call the operator \eqref{WH:eq} a (truncated) Wiener--Hopf operator. 
We are interested in the asymptotics of the trace of the following operator 
difference 
\begin{equation}\label{Dalpha:eq}
D_\a(a, \L; f) := \chi_\L f(W_\a(a; \L)) \chi_\L 
- W_\a(f\circ a; \L),
\end{equation}
as $\a\to\infty$, with some suitably chosen functions $f$. 
The reciprocal parameter $\a^{-1}$ can be interpreted 
as Planck's constant, 
and hence the limit $\a\to\infty$ can be regarded as 
the quasi-classical limit. Sometimes a different point of view is convenient: 
by changing the variables one easily sees that the operator \eqref{Dalpha:eq} 
is unitarily equivalent to $D_1(a, \a\L; f)$, so that the asymptotics  
$\a\to\infty$ can be interpreted as a large-scale limit. 

The second operator on the right-hand side of 
\eqref{Dalpha:eq} 
can be viewed 
as a regularizing term: it makes the operator \eqref{Dalpha:eq} 
trace class even if $f(0)\not = 0$ and 
$\L$ is unbounded, see Condition \ref{domain:cond} 
for precise assumptions on $\L$. On the other hand, 
if $f(0) = 0$, $\L$ is bounded and the symbol 
$a$ decays fast at infinity, 
then $W_\a(f\circ a; \L)$ is trace class itself and  
an elementary calculation shows that 
\begin{equation}\label{weyl:eq}
\tr W_\a(f\circ a; \L) = \frac{\a^d}{(2\pi)^d} |\L| \int f\bigl(a(\bxi)\bigr)d\bxi,
\end{equation} 
where $|\Lambda|$ is the $d$-dimensional Lebesgue measure of $\Lambda$.

Asymptotic properties of $D_\a(a, \L; f)$ depend strongly on the smoothness 
of the symbol $a$. 
For smooth symbols $a$, smooth functions $f$ and smooth bounded domains $\L$, 
the full asymptotic expansion of $\tr D_\a(a, \L; f)$ in powers of 
$\a^{-1}$ was derived by A.~Budylin--V.~Buslaev \cite{BuBu} 
and  H.~Widom \cite{Widom_85}, and we refer to these papers for 
the history of the problem and further references. 
We are concerned only with the leading term asymptotics: 
they have  the form 
\begin{align}\label{bd:eq}
\tr D_\a(a, \L; f) = \a^{d-1}\CB_d(a)  + O(\a^{d-2}), \a\to\infty,
\end{align}
where the coefficient $\CB_d(a)  = \CB_d(a; \p\L, f)$ is defined 
in \eqref{cbd:eq}. 

For symbols $a$ with jump discontinuities 
the asymptotics have a different form. For instance, for $a = \chi_\Om$ with 
a bounded piece-wise region $\Om$, it was found that 
\begin{equation}\label{Widom_conj:eq}
\tr D_\a(\chi_\Om, \L; f) = U\GV_1 \,\a^{d-1}\log \a + 
o(\a^{d-1}\log\a)\,, \;\a\to\infty,
\end{equation}
for a bounded region $\L\subset \R^d$, with explicitly given coefficients 
$U = U(f)$ and 
$\GV_1 = \GV_1(1, \p\L, \p\Om)$, see \eqref{W1:eq} and 
\eqref{U:eq} for the definitions. 
Discontinuous symbols came into prominence after the 
papers by M. E. Fisher and R. E. Hartwig, see e.g. \cite{FH}, 
on determinants of truncated Toeplitz matrices. Ever since, discontinuity 
of the symbol is sometimes referred to as one of the two  
\textit{Fisher--Hartwig singularities}. 
The formula \eqref{Widom_conj:eq} for smooth functions $f$ 
was proved by 
H.~Landau--H.~Widom \cite{Land_Wid},  
H.~Widom \cite{Widom_821} (for $d=1$) 
and by A.V.~Sobolev \cite{Sob, Sob2} (for arbitrary $d\ge 1$). 
These issues have been 
exhaustively studied for the Toeplitz matrices, see e.g. survey 
\cite{Kras} for references.

In the present paper we study the transition from the smooth to discontinuous 
symbols. Precisely, we consider 
smooth symbols $a = a_T$ depending on 
the additional parameter $T>0$, in such a way 
that $a_T(\bxi)\to \chi_\Om(\bxi)$,  
as $T\to 0$ pointwise, with a region $\Om\subset\R^d$, and 
satisfying some mild regularity conditions, see 
\eqref{ata:eq}, \eqref{nablat:eq}. 
 The objective is to investigate 
the asymptotics of $\tr D_\a(a_T, \L; f)$ 
as $T\to 0$ and $\a\to\infty$ simultaneously and independently. 
The sharp bounds for this quantity 
are stated in Theorem \ref{entropy:thm}, and the asymptotic 
results are collected in Theorem \ref{main:thm}. 
The function $f$ 
is not assumed to be globally smooth, but is 
allowed to have finitely many points of non-smoothness. A typical 
example of such a function, with one point of non-smoothness, is 
\begin{align}\label{typf:eq}
f(t) = |t-z|^\g, \ t\in\R,
\end{align} 
with some fixed $z\in\R$, where $\g>0$. The inclusion of 
non-smooth functions $f$ is far from trivial, 
but the relevant tools have been developed earlier (see \cite{Sob_14, LeSpSo_15}), 
and we use them with minor modifications.  

The study of such a two-parameter behaviour seems to be an interesting natural 
problem of asymptotic analysis in its own right. 
Our motivation however comes from the analysis of 
large-scale behaviour of the spatially 
bipartite entanglement entropy 
of free fermions in thermal equilibrium. 
This question amounts to studying the trace 
of the operator \eqref{Dalpha:eq} with a specific choice of the symbol $a_T$ and 
function $f$. 
The symbol is taken to be the \textit{Fermi symbol} 
\begin{equation}\label{positiveT:eq} 
a_T(\bxi) := a_{T, \mu}(\bxi) 
:= \frac{1}{1+ \exp \frac{h(\bxi) - \mu}{T}}\,,\quad \bxi\in\R^d,
\end{equation}
where $T>0$ is the temperature, and $\mu\in\R$ is the chemical potential. 
The function $h = h(\bxi)$ is the free Hamiltonian and we assume that 
$h(\bxi)\to \infty$ as $|\bxi|\to\infty$, so that the \textit{Fermi sea} 
$\Om = \{\bxi\in\R^d: h(\bxi)<\mu\}$ is a bounded set. 
Note that $a_{T, \mu}\to \chi_{\Om},\ T\to 0$, pointwise.

The function\eqref{positiveT:eq} is a 
typical representative of the symbols $a_T$ 
featuring in Theorem \ref{main:thm}, so 
for the sake of discussion in this introduction, 
we assume that $a_T$ is simply given by 
the symbol \eqref{positiveT:eq}. 
The form of the asymptotics in the main theorem  
depends on the 
relation between $\a$ and $T$, the regime $\a T = const$ 
being the critical one.  
If $\a T \le const$, then the asymptotics have 
exactly the form 
\eqref{Widom_conj:eq}, i.e. the same as in the case $a = \chi_\Om$. 
If however $\a T\ge const$, then 
\begin{equation}\label{Widom_conj_T:eq}
\tr D_\a(a_T, \L; f) = U\GV_1 \,\a^{d-1}\log \frac{1}{T} + 
o\biggl(\a^{d-1}\log\frac{1}{T}\biggr)\,, \; T\to 0.
\end{equation}
As proved in Theorem \ref{comparison:thm}, 
the asymptotic formula \eqref{Widom_conj_T:eq} 
can be recast in the form 
\eqref{bd:eq} as follows:
\begin{align}\label{standard:eq}
\tr D_\a(a_T, \L; f) =  \a^{d-1} \CB_d(a_T)  + 
o\biggl(\a^{d-1}\log\frac{1}{T}\biggr)\,, \;T\to 0, \a T\ge const.
\end{align}
Therefore the asymptotic results 
in Theorem \ref{main:thm} do indeed bridge the dichotomy 
between smooth and discontinuous symbols. 

Returning to the large-scale asymptotics 
of the entanglement entropy, 
they follow from Theorem \ref{main:thm} with the symbol \eqref{positiveT:eq}, 
and with the function $f$ which is chosen to be one of the 
\textit{$\g$-R\'enyi entropy functions} $\eta_\g, \g >0,$ 
that are defined 
in \eqref{eta_gamma:eq} and \eqref{eta1:eq}.  
Thus our results provide low-temperature scaling limit of the 
entanglement entropy in all dimensions $d\ge 1$. 
These formulas were announced in the article \cite{LeSpSo2} without underlying 
mathematical details.  
The case of zero temperature, i.e. that of $a = \chi_\Om$ 
was studied in \cite{LeSpSo}. 
The two-parameter asymptotics for the entropy 
were obtained in \cite{LeSpSo_15} for $d = 1$. 
The formulas found there hold for $\a\to\infty$ 
and $\a T\ge const$. In particular, 
these conditions allow the limit $\a\to\infty$, $T = const$. 
On the other hand, Theorem \ref{main:thm} 
always requires $T\to 0$, but allows $\a T \le const$.

The idea is to prove the main result for smooth functions $f$ first. 
In the 
case $\a T \le const$ an elementary argument allows us to 
replace the symbol $a_T$ by its limit $\chi_\Om$, 
and subsequently use the known asymptotic results for discontinuous symbols, 
see \cite{Sob, Sob2}. This produces a formula of the form \eqref{Widom_conj:eq}. 
The case $\a T\ge const$ is substantially more difficult. 
Here we observe that   
different parts of the region $\L$ give different contribution to the trace 
asymptotics. Namely,  
the boundary layer of width $(\a T)^{-1}$ gives the main input into the answer. 
This input is found again by replacing the symbol $a_T$ with the function $\chi_\Om$, and using the results of \cite{Sob, Sob2}.  However, in contrast to the 
$\a T\le const$ case, due to the small size of the boundary layer, 
the resulting asymptotic formula contains 
$\log \frac1{T}$ instead of $\log\a$. The extension of these results to 
non-smooth functions $f$ 
follows the classical idea of asymptotic analysis: 
we approximate $f$ by smooth functions, and for the error we use 
bounds for the trace norm of \eqref{Dalpha:eq} that explicitly depend on the 
function $f$ and the parameter $T$. 
In the abstract setting such 
bounds had been proved in \cite{Sob_14}, and later they were 
used for pseudo-differential operators in \cite{LeSpSo_15} 
for the case $\a T \ge const$. 

The first principal technical ingredient 
is estimates for pseudo-differential operators in 
the Schatten-von Neumann classes $\GS_q, q >0$. 
Since the symbol $a_T$ depends on the extra parameter $T$, 
the main effort goes into controlling 
the dependence of the estimates on the symbol, 
or at least on the parameter $T$. 
Here we rely mostly on the bounds obtained 
in \cite{Sob1} and \cite{LeSpSo_15}, but also derive some new ones, 
see e.g. \eqref{comm_new:eq}. 
Although the main results are concerned with traces and trace norms, 
one should also stress that some intermediate results require 
bounds in the classes $\GS_q$ with $q\in (0, 1)$.
The need for this 
becomes transparent if in  
the operator \eqref{Dalpha:eq} one takes, as an example, the 
function \eqref{typf:eq} with $0<\g <1$.  
 
The second ingredient is the trace asymptotics for  
the operator \eqref{Dalpha:eq} with a discontinuous 
symbol of the type $a = \chi_\Om$. As mentioned earlier, 
these were obtained in \cite{Sob, Sob2}. 
Again, it is crucial that 
these results are uniform in the region $\L$  
in some suitable sense.

Different parts of the proof have different degree of detail. 
In maximal detail we present new arguments, in particular those 
involving explicit 
control of the dependence on the parameter $T$. 
At the same time, the parts of reasoning that repeat previously 
known proofs in new circumstances, are just sketched and sometimes, omitted. 

The plan of the paper is as follows: 
in Section \ref{main:sect} we provide some basic information 
on Schatten-von Neumann classes, including the useful 
$q$-triangle inequality \eqref{qtriangle:eq}, 
and state the main results, 
see Theorems \ref{entropy:thm} and \ref{main:thm}. 
The whole of Section \ref{entropy:sect}
 is devoted to applications of 
 the main theorems to the study of various 
 entropies of fermionic systems. 
Some elementary estimates for smooth 
functions of self-adjoint operators are presented in Section 
\ref{smooth:sect}. 
In Sections \ref{WH:sect} 
and \ref{at:sect} we collect the necessary 
Schatten-von Neumann estimates for pseudo-differential operators 
and Wiener-Hopf operators, and prove Theorem \ref{entropy:thm}.  
Section \ref{discont:sect}  
contains preliminary information about trace asymptotics 
for discontinuous symbols. 
The main theorems are proved in Sections 
\ref{proof1:sect} and \ref{proof2:sect}. 
Rewriting the results for $\a T\gtrsim 1$ 
in the form \eqref{standard:eq} 
takes a lot of technical work, 
which is done in Sections \ref{compar:sect} 
and \ref{coeff:sect}. 

Throughout the paper we adopt the following convention.
For two non-negative numbers (or functions) 
$X$ and $Y$ depending on some parameters, 
we write $X\lesssim Y$ (or $Y\gtrsim X$) if $X\le C Y$ with 
some positive constant $C$ independent of those parameters.
If $X\lesssim Y$ and $X\gtrsim Y$, then we write $X\asymp Y$. 
For example, $\a T\asymp 1$ 
means that $c\le \a T\le C$ with some constants $C, c$, independent 
of $\a$ and $T$. 
To avoid confusion we often make explicit comments on the nature of 
(implicit) constants in the bounds. 

The notation $B(\bz, R)\subset\R^d$, $\bz\in\R^d$, $R>0$,
is used for the open ball of radius $R$, centred at the point 
$\bz$. 

For any vector $\mathbf v\in\R^n$, $n=1, 2, \dots$ we use the standard 
notation $\lu \mathbf v\ru = \sqrt{1+|\mathbf v|^2}$.

\textbf{Acknowledgements.} The author is grateful 
to W. Spitzer for useful remarks. 

Thanks also go to the anonymous referee for a number of useful 
suggestions. 

This work was supported by EPSRC grant EP/J016829/1.

\section{Preliminaries. Main result}\label{main:sect}
 
 \subsection{Schatten-von Neumann classes}
 We use some well-known facts about 
Schatten--von Neumann operator ideals $\GS_q, q>0$. 
Detailed information on these ideals can be found e.g. in \cite{BS,GK,Pie,Simon}. 
We shall point out only some basic facts.
For a compact operator $A$ on a separable Hilbert space $\CH$ denote 
by $s_n(A), n = 1, 2, \dots$ its singular values, 
i.e., the eigenvalues of the operator $|A| := \sqrt{A^*A}$. 
We denote the identity operator on $\CH$ by $I$. 
The Schatten--von Neumann ideal $\GS_q$, $q>0$ consists of all 
compact operators $A$, for which 
\begin{equation*}
	\|A\|_q := \biggl[\sum_{k=1}^\infty s_k(A)^q\biggr]^{\frac{1}{q}}<\infty.
\end{equation*}
If $q\ge 1$, then the above functional defines a norm. 
If $0<q <1$, then it is a quasi-norm. 
There is nevertheless a convenient analogue of the triangle inequality, which is called 
the \textit{$q$-triangle inequality:}
\begin{equation}\label{qtriangle:eq} 
\|A_1+A_2\|_q^q\le \|A_1\|_q^q + \|A_2\|_q^q, \ \ 
A_1, A_2\in\GS_q, 0 < q \le 1,
\end{equation}
see \cite{Rot} and also \cite{BS}. Thus $\|A\|_q$ is sometimes called 
a \textit{$q$-norm}. 
Note also the H\"older inequality
\begin{equation*}
\|A_1 A_2\|_q\le \|A_1\|_{q_1} \cdot \|A_2\|_{q_2}, 
\ \ q^{-1} = q_1^{-1} + q_2^{-1}\,,\ \ 0 < q_1,q_2\le\infty.
\end{equation*}
Further on we need some $\GS_q$-estimates for functions of 
self-adjoint operators that were established in \cite{Sob_14}. 
As indicated in the Introduction, we are interested in functions 
that lose smoothness at finitely many points. Without loss of generality, for 
almost all estimates we may assume that $f$ has only one non-smoothness point. 
Below $\chi_R$ denotes the indicator function of 
the interval $(-R, R)$, $R>0$. 
We impose the following condition. 
 
\begin{cond}\label{f:cond}
For some integer $n \ge 1$ the function 
	$f\in\plainC{n}(\R\setminus\{ t_0 \})\cap\plainC{}(\R)$ satisfies the 
	bound 
	\begin{equation}\label{fnorm:eq}
	\1 f\1_n = 
	\max_{0\le k\le n}\sup_{t\not = t_0} 
	|f^{(k)}(t)| |t-t_0|^{-\g+k}<\infty
	\end{equation}
	with some $\g >  0$, 
	and is supported on the 
	interval $[t_0-R, t_0+R]$ with some $R>0$.  
\end{cond}
A function $f$ satisfying \eqref{fnorm:eq} with 
$ n = 1$ is H\"older-continuous:
\begin{equation}\label{hf:eq}
|f(t_1)-f(t_2)|
\le 2 \1 f\1_1|t_1-t_2|^\varkappa,\ \varkappa = \min \{\g, 1\},
\end{equation}
for all $t_1, t_2\in\R$.
In what follows, all the bounds 
involving functions from Condition \ref{f:cond}, 
are uniform in $t_0$, and 
contain explicit dependence on the quantity 
\eqref{fnorm:eq}, and on the radius $R$.

Now we can quote one abstract result following 
from \cite[Theorem 2.10]{Sob_14}.
    
\begin{prop}\label{Szego1:prop} 
Suppose that $f$ satisfies Condition 
\ref{f:cond} with some $\g >0$,  $n\ge 2$ and 
some $t_0\in\R$, $R\in (0, \infty)$. 
Let $q$ be a number such that $(n-\s)^{-1} < q\le 1$ with some 
number $\s \in (0, 1]$, $\s < \g$. 
Let $A$ be a bounded self-adjoint operator 
and let $P$ be an orthogonal projection such that 
$PA(I-P)\in \GS_{\s q}$. Then 	
\begin{equation}\label{Szego:eq}
	\|f(PAP)P - P f(A)\|_q
	\lesssim \1 f\1_n R^{\g-\s} \|PA(I-P)\|_{\s q}^\s,
\end{equation}
with an implicit constant independent 
of the operators $A, P$, the 
function $f$, and the parameters $R, t_0$.  
\end{prop}  

Later in the proofs we apply this 
proposition to the operator \eqref{Dalpha:eq}.   
  
We also need a version of Proposition 
\ref{Szego1:prop} for smooth functions $f$, see 
\cite[Corollary 2.11]{Sob_14}.

\begin{prop}\label{smoothP:prop}
	Suppose that $g\in \plainC{n}_0(-r, r)$, 
	with some $r >0$ and $n\ge 2$. 
	Assume $q\in (0, 1]$ and $\s\in (0, 1]$ are such that $(n-\s)^{-1} < q \le 1$.  
	Let the operator $A$ and orthogonal projection $P$ be as in Proposition 
	\ref{Szego1:prop}. 
	Then 
	\begin{equation*}
	\|g(PAP)P - P g(A)\|_q
	\lesssim \|g\|_{\plainC{n}} \|PA(I-P)\|_{\s q}^\s,
	\end{equation*}
	with an implicit constant 
independent of the operator $A$, projection $P$ and the function $g$.  
\end{prop}

\subsection{The domains and regions}\label{domains:subsect} 
We always assume that $d\ge 2$. 
We say that $\L$ is a basic Lipschitz 
(resp. basic $\plainC{m}$, $m = 1, 2, \dots$) domain, if 
there is a Lipschitz (resp. $\plainC{m}$) 
function $\Phi = \Phi(\hat\bx)$, $\hat\bx\in\R^{d-1}$, such that 
with a suitable choice of Cartesian coordinates $\bx = (\hat\bx, x_d)$, 
the domain $\L$ is represented as  
\begin{equation}\label{basic_d:eq}
\L = \{\bx\in\R^d: x_d > \Phi(\hat\bx)\}. 
\end{equation} 
We use the notation $\L = \G(\Phi)$. 
The function $\Phi$ is assumed to be globally Lipschitz, i.e. 
the constant 
\begin{equation}\label{MPhi:eq}
M_{\Phi} = \underset{\hat\bx\not=\hat\by }
\sup\ \frac{|\Phi(\hat\bx) - \Phi(\hat\by)|}{|\hat\bx-\hat\by|},
\end{equation}
 is finite. 
 \textit{Throughout the paper, all 
 estimates involving basic Lipschitz domains 
$\L = \G(\Phi)$, are uniform in the number $M_\Phi$}.  
 
A domain (i.e. connected open set) 
is said to be Lipschitz (resp. $\plainC{m}$) if locally 
it coincides with some basic Lipschitz 
(resp. $\plainC{m}$) domain. 
We call $\L$ a Lipschitz (resp. $\plainC{m}$) 
region if $\L$ is a union of 
finitely many Lipschitz (resp. $\plainC{m}$) 
domains such that their closures are pair-wise disjoint.  

A basic Lipschitz 
domain $\L=\G(\Phi)$ is said to be piece-wise $\plainC{m}$  
with some $m = 1, 2, \dots$, if  the function 
$\Phi$ is $\plainC{m}$-smooth away from a collection of 
finitely many  $(d-2)$-dimensional 
Lipschitz surfaces $L_1, L_2, \dots\subset \R^{d-1}$. 
By $(\p\L)_{\rm s}\subset\p\L$ we denote the set of points 
where the $\plainC{m}$-smoothness of the surface $\p\L$ may break down. 
A piece-wise $\plainC{m}$ region $\L$ and the set $(\p\L)_{\rm s}$ 
for it are defined in the obvious way. 
An expanded version of these definitions 
can be found in \cite{Sob1}, \cite{Sob2}, and here we omit the details. 


The minimal assumptions on the sets
featuring in this paper are laid out in the following condition. 

\begin{cond}\label{domain:cond}
The set $\L\subset \R^d, d\ge 2$, 
is a Lipschitz region, 
and  either $\L$ or 
	$\R^d\setminus\L$ is bounded.
\end{cond}

Some results, including the main asymptotic formulas in Theorem \ref{main:thm}, 
require extra smoothness of $\L$.  
Note that if $\L$ is a Lipschitz (or $\plainC{m}$) region, then so is 
the interior of $\R^d\setminus \L$.

\subsection{The main result}\label{at:subsect} 
We study the operator $D_\a(a_T, \L; f)$, see 
\eqref{Dalpha:eq} for the definition, with the 
symbol $a_T$ approximating the indicator function 
of an open set $\Om\subset\R^d$. 
To state the precise definition of $a_T$ 
denote 
\begin{align}\label{distance:eq}
\rho(\bxi) = \dist(\bxi, \p\Om),\  
\tilde \rho(\bxi) = \min\{\rho(\bxi), 1\},\ 
\lu t\ru  = \sqrt{1+{t^2}},\ t\in \R.
\end{align} 

\begin{cond}\label{at:cond}
Let $\Om\subset \R^d$, $d\ge 2,$ 
satisfy Condition \ref{domain:cond}. 
The symbol $a_T = a_{T, \Om}\in\plainC\infty(\R^d)$, 
depending on the set $\Om$ and parameter $T\in (0, T_0]$, is a function  
satisfying the properties
\begin{align}\label{ata:eq}
|a_T(\bxi) - \chi_\Om(\bxi)|\lesssim \lu \rho(\bxi)T^{-1}\ru^{-\b}
, \ \ \b > d, 
\end{align}
and 
\begin{align}\label{nablat:eq}
|\nabla^m a_T(\bxi)|\lesssim 
(T+\tilde\rho(\bxi))^{-m} 
\lu \rho(\bxi)T^{-1}\ru^{-\b},\  
m = 1, 2, \dots.
\end{align}
The implicit constants are independent 
of $T$, but may depend on $T_0$ and region $\Om$. 
\end{cond}

Although $a_{T, \Om}$ depends on two parameters, i.e. the number 
$T\in(0, T_0]$ and the region 
$\Om\subset\R^d$, we usually omit the dependence on $\Om$, since $\Om$ is fixed.

The main results of this paper are contained in the next 
two theorems.

\begin{thm}\label{entropy:thm}  
Let $d\ge 2$. Suppose that 
$\L\subset\R^d$ and $\Om\subset\R^d$  satisfy Condition \ref{domain:cond}.
Let $a_T = a_{T, \Om}\in\plainC\infty(\R^d)$ be a real-valued symbol depending 
on the parameter $T: 0 < T\lesssim 1$, and 
satisfying Condition \ref{at:cond}. Suppose also that $\a\gtrsim 1$. 

Suppose that $f$ satisfies Condition \ref{f:cond} with $n=2$ 
and some $\g>0$. 
%
%
If $\b > \max\{d, d\g^{-1}\}$, 
then for any $\s\in (d\b^{-1}, \min\{\g, 1\})$, we have 
\begin{equation}\label{entropy:eq}
	\bigl\| 
	D_\a(a_{T}, \L; f)\bigr\|_1
	\lesssim
	R^{\g-\s} \a^{d-1} \log \biggl( 
	\min \biggl\{\a, \frac{1}{T}\biggr\}+1\biggr)   
	\1 f\1_2,
	\end{equation}
	with an implicit constant independent of $T, R, t_0, \a$ and the function $f$. 
\end{thm}

For the main asymptotic result we
need some more notation. 
For any two Lipschitz regions $\L$, $\Om$ and a continuous function 
$b = b(\bx, \bxi)$ define the quantity
\begin{equation}\label{W1:eq}
\GV_1(b) = \GV_1(b; \p\L, \p\Om) = \frac{1}{(2\pi)^{d+1}}
\underset{\p\L}\int\ \underset{\p\Om}
\int b(\bx, \bxi)|\bn_{\bxi}\cdot \bn_{\bx}| 
d S_{\bxi} d S_{\bx},
\end{equation}
where $\bn_{\bx}, \bn_{\bxi}$ are the exterior unit normals to the surfaces $\p\L$ 
and $\p\Om$ at the points $\bx$ and $\bxi$ respectively. 
For a H\"older-continuous function $f: \mathbb C\mapsto \mathbb C,$ 
define the integral
\begin{equation}\label{U:eq}
 	U(f) 
 	= \int_0^1 \frac{f(1-t) 
 		- (1-t)f(1) - t f(0)}{t(1-t)} dt.
 \end{equation} 
Now we can describe the asymptotics of $D_\a(a_T, \L; f)$.

\begin{thm}\label{main:thm} 
Let $d\ge 2$. 
Suppose that 
$\L\subset\R^d$ and $\Om\subset\R^d$ satisfy Condition \ref{domain:cond}.  
Suppose also 
that $\L$ is piece-wise $\plainC1$, 
and $\Om$ is piece-wise $\plainC3$.

Let $X = \{z_1, z_2, \dots, z_N\}\subset \R$, $N <\infty$, be a collection of 
points on the real line. 
Suppose that $f\in\plainC2(\R\setminus X)$ is a function such that 
in a neighbourhood of each point $z\in X$ 
it satisfies the bound
\begin{equation}\label{fmain:eq}
|f^{(k)}(t)|\le C_k |t - z|^{\g-k}, \ k = 0, 1, 2, 
\end{equation} 
with some $\g>0$. 

Suppose that $a_T = a_{T, \Om}\in\plainC\infty(\R^d)$, $0<T\lesssim 1$, 
is a real-valued symbol 
as in Theorem \ref{entropy:thm}.

Then 
\begin{align}\label{main_gr:eq}
\underset{ \substack{T\to 0\\\a T\gtrsim 1}}\lim\ 
\frac{1}{\a^{d-1}\log \frac{1}{T}}
\tr D_\a( a_{T}, \L;  f)  
=  U(f)\GV_1( 1; \p\L, \p \Om),
\end{align} 
and
\begin{align}\label{main_le:eq}
\underset{ \substack{\a\to\infty\\ \a T\lesssim 1}}\lim\ 
\frac{1}{\a^{d-1}\log \a}
\tr D_\a( a_{T}, \L;  f)  
=  U(f)\GV_1( 1; \p\L, \p \Om).
\end{align} 
\end{thm}

\begin{rem}\label{boundom:rem}
Note that in Theorems \ref{entropy:thm} and \ref{main:thm} 
either one or both regions 
$\L$ or $\Om$ are allowed to be unbounded, as long as the complements  
$\Om^c$ and/or $\L^c$ are bounded.  Nevertheless,  
it suffices to prove Theorems \ref{entropy:thm} and \ref{main:thm} 
for a bounded $\Om$ only. Indeed, suppose that $\Om$ is a Lipschitz region such that 
$\Om^c$ is bounded. Denote 
\begin{align*}
b_T = b_{T, \Om^c} = 1-a_{T, \Om}, \ \ g(t) = f(-t+1).
\end{align*}
The symbol $b_T$ satisfies Condition \ref{at:cond} with $\Om^c$ instead of $\Om$.
The function $g$ satisfies 
Condition \ref{f:cond} with $1-t_0$ instead of $t_0$ 
(for Theorem \ref{entropy:thm}), 
or it satisfies the bound \eqref{fmain:eq} for all $z\in \tilde X$ 
with $\tilde X = 1-X$ (for Theorem \ref{main:thm}). Moreover, 
\begin{align*}
f\bigl(W_\a(a_T; \L)\bigr)\chi_\L 
= g\bigl(W_\a(b_T; \L)\bigr)\chi_\L,\ 
 f\bigl(\op_\a(a_T)\bigr) = g\bigl(\op_\a(b_T)\bigr),
\end{align*}
so that $D_\a(a_{T, \Om}, \L; f) = D_\a(b_{T, \Om^c}, \L; g)$. Thus 
Theorems \ref{entropy:thm} and \ref{main:thm} 
for the symbol $a_T$ and function $f$  
follow from themselves for the symbol $b_T$ and function $g$. 

This observation has no bearing on 
many intermediate estimates as 
our argument usually goes through equally 
well for bounded or unbounded $\Om$. Nevertheless, 
it will be useful to us since some of the results that we borrow from  
the literature, have been formally proved for bounded $\Om$ only. 
\end{rem}

\begin{rem}
If the regions $\L$ and $\Om$ are bounded, and $f(0) = 0$, then the second operator 
on the right-hand side of \eqref{Dalpha:eq} is trace class, 
and its trace is easily found from the formula (see \eqref{weyl:eq}) 
\begin{align*}
\tr W_\a(f\circ a_T; \L) = \frac{\a^d}{(2\pi)^d} \underset{\L}\int 
\underset{\R^d}\int f(a_T(\bxi)) d\bxi d\bx 
=: \a^d \GW_0(a_T\chi_\L; f),
\end{align*} 
and thus the formulas \eqref{main_gr:eq} 
and \eqref{main_le:eq} give two-term asymptotics 
for $\tr f(W_\a)$. For instance, \eqref{main_gr:eq} 
can be rewritten as 

\begin{align}\label{Weyl:eq}
\tr f\bigl(W_\a(a_T; \L)\bigr) =  &\ \a^d \GW_0(a_T\chi_\L; f)\notag\\[0.2cm] 
&\ + \a^{d-1}\log\frac{1}{T} \ U(f)\GV_1(1; \p\L, \p\Om) 
+ o\biggl(\a^{d-1}\log\frac{1}{T}\biggr),\ \a T\gtrsim 1, T\to 0.
\end{align}
\end{rem}

In this paper we 
discuss two possible realizations of the symbol $a_T$. 
The first example is given in the next subsection. 

\subsection{Anti-Wick quantization} 
Let $\Om$ be a Lipschitz region. 
One way to approximate $\chi_\Om$ by a smooth function 
is to use the following standard procedure. 
Let $\z\in\plainC\infty_0(\R^d)$ be a non-negative function 
with a support in the unit ball $B(\bold0, 1)$, such 
that $\|\z\|_{\plainL1} = 1$. Define 
 \begin{align}\label{smooth:eq}
 a_T(\bxi) = \int_\Om \z_T(\bxi-\boldeta) d\boldeta,\ \z_T(\bxi) = 
 \frac{1}{T^d}\z\biggl(\frac{\bxi}{T}\biggr).
 \end{align}
It is straightforward that $a_T(\bxi) - \chi_\Om(\bxi) = 0$ 
for all $\bxi$ with the property 
$\rho(\bxi)>T$, and $|\nabla^m a_T(\bxi)|\lesssim T^{-m}$, 
$m= 1, 2, \dots,$ if $\rho(\bxi)\le T$. 
Thus $a_T$ satisfies \eqref{ata:eq} and \eqref{nablat:eq} for any $\b >0$, and 
hence Theorems \ref{entropy:thm} and 
\ref{main:thm} hold. 

Clearly, the smoothing 
function $\z$ does not need to have a compact support. One can 
pick, for instance $\z(\bxi) = c_d\exp(-|\bxi|^2)$ 
with the appropriate normalizing coefficient 
$c_d$. In this case the conditions 
\eqref{ata:eq} and \eqref{nablat:eq} are also 
readily verified.  With this 
choice of the smoothing $\z$, the symbol 
resembles the definition of the anti-Wick quantization of 
a pseudo-differential operator, see e.g. \cite[Ch. 4, \S 24]{Sh}.
Recall that if the anti-Wick symbol is given by 
$b(\bx, \bxi)$, then 
the ``quasi-classical" Weyl symbol $b^{\rm W}$ 
of the same operator is found via the formula 
\begin{align*}
b^{\rm W}_T(\bx, \bxi)
= \frac{1}{(\pi T^2)^d} \int \exp
\biggl({-\frac{|\bx-\by|^2}{T^2} - \frac{|\bxi-\boldeta|^2}{T^2}}\biggr)
b(\by, \boldeta) d\by d\boldeta.
\end{align*}
Since \eqref{smooth:eq} does not contain any dependence on the spatial variable, 
the symbol $a_T$ in fact is the Weyl symbol associated with the anti-Wick symbol 
$\chi_\Om(\bxi)$. At the same time,  the indicator function $\chi_\L$ 
in the definition \eqref{WH:eq} is not smoothed-out. Thus the operator 
$W_\a(a_T; \L)$ can be loosely described as a ``partial" 
anti-Wick quantization. Therefore 
it seems appropriate to compare our results, e.g. 
formula \eqref{Weyl:eq}, with the known asymptotic formulas for anti-Wick operators. 
These operators 
were introduced 
(see \cite{Berezin}) and subsequently extensively studied 
in the case $T = \a^{-1/2}$. We do not intend to discuss 
these studies in detail, and cite only the latest asymptotic result in this area, 
found in 
\cite{Oldfield}. This paper was concerned with potentially 
discontinuous anti-Wick symbols of the form 
$ b(\bx, \bxi) = \chi_M(\bx, \bxi) \tilde b(\bx, \bxi)$ where $M\subset\R^{2d}$ 
is a bounded smooth domain, and $\tilde b$ is a smooth symbol. 
For the sake of precision one should say that the paper \cite{Oldfield} considered 
\textit{generalized} anti-Wick-operators, i.e. those with arbitrary 
decreasing smoothing functions $\z$, but we do not elaborate on this point here. 
The main result of \cite{Oldfield} states that for smooth functions $f$, 
such that $f(0) = 0$,  
one has 
\begin{align*}
\tr f(\op_\a(b^W_T)) = \a^d \GW_0(b^W_T; f) + O(\a^{d-1}) \ \ \textup{with}\ 
\ T = \frac{1}{\sqrt{\a}},\ \ \textup{as}\ \ \a\to\infty.
\end{align*}
This formula, in contrast to \eqref{Weyl:eq}, does not have a $\log$-term.

\section{
The Fermi symbol. Entanglement entropy and local entropy}
\label{entropy:sect}

This section is focused on our second example of the symbol $a_T$. 
%
%
As mentioned earlier, the asymptotical problems studied in this 
paper are partly motivated by the study of entanglement entropy of free 
fermions both for positive and zero temperature $T$. In particular, 
for the  positive temperature $T>0$ the relevant symbol $a_T$ is the Fermi symbol 
\eqref{positiveT:eq} with temperature $T>0$ and chemical potential 
$\mu >0$. The parameter $\mu$ 
is always assumed fixed. 
The assumptions on the function $h = h(\bxi)$ are as follows:

\begin{cond} \label{h:cond}
\begin{enumerate}
\item
The function $h\in\plainC\infty(\R^d)$ is real-valued,  
 and for sufficiently large $\bxi$ we have 
\begin{equation}\label{equih:eq}
h(\bxi)\gtrsim |\bxi|^{\b_1}, 
\end{equation}
with some constant $\b_1>0$. Moreover, 
\begin{equation}\label{derivh:eq}
|\nabla^n h(\bxi)| \lesssim (1+|\bxi|)^{\b_2}, \ n= 1, 2, \dots,\ \ 
\forall \bxi\in\R^d,
\end{equation}
with some $\b_2\ge 0$. 
\item  
On the set   
$S = S_\mu = \{\bxi\in\R^d: h(\bxi) = \mu\}$ the condition 
\begin{equation}\label{nondeg:eq}
\nabla h(\bxi)\not = 0,\ \forall\bxi\in S. 
\end{equation}
is satisfied. 
\item
The set $\Om_\mu = \{\bxi\in\R^d: h(\bxi) < \mu\}$ has finitely many connected components. 
\end{enumerate}
\end{cond}

Because of \eqref{equih:eq} 
the set $\Om = \Om_\mu$ is bounded, i.e. $\Om\subset B(\bold0, R_0)$ 
with some $R_0>0$.   
Furthermore, due to the condition \eqref{nondeg:eq}, the set 
$S_\mu$ is a $\plainC\infty$-surface (called \textit{the Fermi surface}), 
and we have the bounds:
 \begin{equation}\label{twosided:eq}
 |h(\bxi)-\mu|
 \asymp \rho(\bxi), \ \forall \bxi\in B(\bold0, R_0).
 \end{equation}
By virtue of \eqref{equih:eq} we also have the lower bound 
\begin{align*}
h(\bxi) - \mu\gtrsim |\bxi|^{\b_1},\ |\bxi|\ge R_0.
\end{align*}
Note the straightforward bound:
\begin{align*}
|a_{T, \mu}(\bxi) - \chi_\Om(\bxi)|\le \exp\biggl(
- \frac{|h(\bxi)-\mu|}{T}
\biggr),\ \bxi\in\R^d.
\end{align*}
This allows us to extend the definition \eqref{positiveT:eq} to 
$T=0$:
\begin{align*}
a_{0, \mu}(\bxi) = \underset{T\to 0}\lim a_{T, \mu}(\bxi) = \chi_\Om(\bxi),\ 
\hbox{a.e.} \ \bxi\in\R^d.
\end{align*}
As pointed out in \cite[Section 8]{LeSpSo_15}, 
as a consequence of \eqref{equih:eq}, \eqref{derivh:eq} 
and \eqref{twosided:eq}, 
we also have
\begin{align*}
|\nabla^n a_{T, \mu}(\bxi)| \lesssim &\ (T+\rho(\bxi))^{-n}
\exp\biggl(-c_1\frac{\rho(\bxi)}{T}
\biggr),\ n = 1, 2, \dots, \ \forall |\bxi|\le R_0,\\
|\nabla^n a_{T, \mu}(\bxi)| \lesssim &\  
\exp\biggl(-c_1\frac{|\bxi|^{\b_1}}{T}
\biggr),\ n = 1, 2, \dots, \ \forall |\bxi|\ge R_0,
\end{align*}
with some positive constant $c_1$.
Thus the symbol \eqref{positiveT:eq} satisfies 
the bounds \eqref{ata:eq} and \eqref{nablat:eq} for any $\b >0$. 
Consequently, the trace $\tr D_\a(a_{T, \mu}, \L; f)$ 
satisfies \eqref{main_gr:eq} and \eqref{main_le:eq}. 
In order to study the entropy we use these asymptotic formulas with
the \textit{$\g$-R\'enyi entropy function} 
$\eta_\g: \R\mapsto [0, \infty)$ 
defined for all $\g >0$ as follows. If $\g\not = 1$, then
\begin{equation}\label{eta_gamma:eq}
\eta_\gamma(t) := \left\{\begin{array}{ll} 
\frac{1}{1-\gamma} \log\big[t^\gamma + (1-t)^\gamma\big]& 
\mbox{ for }t\in(0,1),\\[0.2cm]
0&\mbox{ for }t\not\in(0,1),\end{array}\right.
\end{equation}
and for $\gamma=1$ (the von Neumann case) it is defined as the limit 
\begin{equation}\label{eta1:eq}
\eta_1(t) := \lim_{\gamma\to1} \eta_\gamma(t) = \left\{\begin{array}{ll} -t \log(t) -(1-t)\log(1-t)& \mbox{ for } t\in(0,1),\\[0.2cm]
0& \mbox{ for }t\not\in(0,1).\end{array}\right.
\end{equation} 
For $\g\not = 1$ the function $\eta_\g$ satisfies condition \eqref{fmain:eq} 
with $\g$ replaced with $\varkappa = \min\{\g, 1\}$, and with $X=\{0, 1\}$. 
The function $\eta_1$ satisfies \eqref{fmain:eq} with an 
arbitrary $\g\in (0, 1)$, 
and the same set $X$. 

We begin with reminding 
definitions of the entropies in the form given in \cite[Section 10]{LeSpSo_15}. 
If $\L\subset\R^d$ is bounded, then \textit{the local (thermal) $\g$-R\'enyi entropy} 
of the equilibrium state at temperature $T\ge 0$ and chemical potential 
$\mu\in\R$ is defined as 
\begin{align}\label{thermal:eq}
\mathrm{S}_\g(T, \mu; \L):=\tr\bigl[
\eta_\g(W_1(a_{T, \mu}; \L))
\bigr].
\end{align}
For arbitrary $\L\subset\R^d$ 
we define the  
\textit{the $\g$-R\'enyi entanglement entropy (EE)} with respect 
to the bipartition 
$\R^d = \L \cup (\R^d\setminus\L)$, as 
\bee \label{def:EE}
\mathrm{H}_\g(T,\mu; \L) 
:= \tr D_1(a_{T,\mu},\L;\eta_\g) 
+ \tr D_1(a_{T,\mu},\R^d\setminus\L;\eta_\g). 
\ene 
These entropies were studied in \cite{LeSpSo} (for $T = 0$) 
and \cite{LeSpSo_15} (for $T >0$).
In particular, in \cite{LeSpSo_15} it was shown that for any $T>0$ 
the EE is finite, if $\L$ satisfies Condition \ref{domain:cond}. 

We are interested in the behaviour of 
the above quantities when $T\to 0$ and $\L$ is replaced with $\a\L$, 
with a large scaling parameter $\a$.  
The next theorem 
establishes sharp bounds for the entropies 
\eqref{thermal:eq} and \eqref{def:EE}.

\begin{thm}\label{EE_bound:thm}
Let $d\ge 2$. Suppose that 
$\L$ satisfies Condition \ref{domain:cond}. 
Suppose that $0<T\lesssim 1$ and $\a\gtrsim 1$.  
Then the $\g$-R\'enyi entanglement entropy
satisfies  
\bee \label{EE bound}
|\mathrm{H}_\g(T,\mu;\a\L)| 
\lesssim \a^{d-1} \log\bigl(\min\bigl\{\a, T^{-1}\bigr\}+1\bigr).
\ene
If $\L$ is bounded, then the local $\g$-R\'enyi entropy satisfies
\begin{align}\label{sgam:eq}
\bigl|\mathrm{S}_\g(T, \mu; \a\L) 
- \a^d s_\g(T, \mu)|\L|\bigr|\lesssim \a^{d-1}
\log\bigl(\min\bigl\{\a, T^{-1}\bigr\}+1\bigr),
\end{align}
where 
\begin{align*}
s_\g (T, \mu) := \frac{1}{(2\pi)^d} 
\int \eta_\g(a_{T, \mu}(\bxi)) d\bxi.
\end{align*}
The constants in \eqref{EE bound} and \eqref{sgam:eq} 
are independent of $\a$ and $T$, 
but may depend on the function $h$, parameter $\mu$ and the region $\L$. 
\end{thm}

The coefficient $s_\gamma(T,\mu)$ 
is called the \textit{$\g$-R\'enyi entropy density} 
(cf.~\cite{LeSpSo}). For $\a T\gtrsim 1$ 
the bounds \eqref{EE bound} and \eqref{sgam:eq} 
were derived in \cite{LeSpSo_15}. 

\begin{proof}[Proof of Theorem \ref{EE_bound:thm}]
In view of additivity of the operator $D_\a(a, \L; f)$ in the functional 
parameter $f$, the bound \eqref{EE bound} follows from Theorem \ref{entropy:thm}. 
  
In order to prove \eqref{sgam:eq}, we rewrite \eqref{thermal:eq}:
\begin{align*}
\mathrm{S}_\g(T,\mu;\a\L) 
= \tr[\chi_\L \eta_\g(\op_\a(a_{T,\mu}))\chi_\L] 
+ \tr D_\a(a_{T,\mu},\L; \eta_\g).
\end{align*}
For the first trace we use the identity \eqref{weyl:eq}, 
and for the second one -- 
the bound \eqref{EE bound}. 
This yields \eqref{sgam:eq}.
\end{proof}

The next theorem establishes the 
asymptotic behaviour of the local entropy and of the EE. 

\begin{thm}\label{EE:thm} 
Suppose that $\L\subset\R^d, d\ge 2$, 
satisfies Condition \ref{domain:cond}, and 
is piecewise $\plainC1$. Then the EE satisfies
\begin{align}\label{mainEE_gr:eq}
\underset{ \substack{T\to 0\\\a T\gtrsim 1}}\lim\ 
\frac{1}{\a^{d-1}\log \frac{1}{T}}
\mathrm{H}_\g(T, \mu; \a\L)  
=  2{\pi^2}\,\frac{1+\g}{6\g}\GV_1( 1; \p\L, \p \Om),
\end{align} 
and
\begin{align}\label{mainEE_le:eq}
\underset{ \substack{\a\to\infty\\ \a T\lesssim 1}}\lim\ 
\frac{1}{\a^{d-1}\log \a}
\mathrm{H}_\g(T, \mu; \a\L)  
=  2{\pi^2}\,\frac{1+\g}{6\g}\GV_1( 1; \p\L, \p \Om).
\end{align} 
If the region $\L$ is bounded, then 
the local entropy satisfies
\begin{align}\label{mainLE_gr:eq}
\underset{ \substack{T\to 0\\\a T\gtrsim 1}}\lim\ 
\frac{1}{\a^{d-1}\log \frac{1}{T}}
\biggl(\mathrm{S}_\g(T, \mu; \a\L) - \a^d s_\g(T, \mu)|\L|\biggr)  
=  {\pi^2}\,\frac{1+\g}{6\g}\GV_1( 1; \p\L, \p \Om),
\end{align} 
and
\begin{align}\label{mainLE_le:eq}
\underset{ \substack{\a\to\infty\\ \a T\lesssim 1}}\lim\ 
\frac{1}{\a^{d-1}\log \a}
\biggl(\mathrm{S}_\g(T, \mu; \a\L) - \a^d s_\g(T, \mu)|\L|\biggr)  
=  {\pi^2}\,\frac{1+\g}{6\g}\GV_1( 1; \p\L, \p \Om).
\end{align} 
\end{thm}

\begin{proof}
Formulas \eqref{mainEE_gr:eq} and \eqref{mainEE_le:eq} 
follow from 
\eqref{main_gr:eq} and \eqref{main_le:eq} respectively 
upon observing (cf. \cite{LeSpSo}) that 
\[ 
U(\eta_\g) 
= \int_0^1 \frac{\eta_\g(t)}{t(1-t)} dt = {\pi^2}\,\frac{1+\g}{6\g}.
\]
Formulas \eqref{mainLE_gr:eq} and \eqref{mainLE_le:eq} also follow 
from \eqref{main_gr:eq} and \eqref{main_le:eq}, and from \eqref{weyl:eq}.
\end{proof}

For $d= 1$ and $\a T\gtrsim 1$ Theorem \ref{EE:thm} was proved in \cite{LeSpSo_15}. 
We also re-iterate that the formulas 
above agree with the large-scale asymptotics of the entropies 
$\mathrm{H}_\g$ and $\mathrm{S}_\g$ 
for the zero temperature case, which were found in \cite{LeSpSo}. 
   
\section{Smooth functions of self-adjoint operators}
\label{smooth:sect}

\subsection{The Helffer-Sj\"ostrand formula} 
When studying functions of self-adjoint operators 
we rely on  
the Helffer-Sj\"ostrand formula (see \cite{HelSjo}) 
which holds for arbitrary operator $X=X^*$  
and arbitrary smooth function $f\in\plainC{n}_0(\R), n\ge 2$:
\begin{equation}\label{HS:eq}
f(X) = \frac{1}{\pi} \iint \mathcal Z(x, y)(X-x-iy)^{-1} dx dy,\
\mathcal Z(x, y) = \frac{\p}{\p \bar z} \tilde f(x, y),	
\end{equation}
where $\tilde f =  \tilde f(x, y)$ 
is a quasi-analytic extension of the function $f$, see 
\cite[Ch. 2]{EBD}.  
A quasi-analytic extension of 
$f\in \plainC{n}(\R)$ is a 
$\plainC1(\R^2)$-function $\tilde f$, such that 
$f(x) = \tilde f(x, 0)$ and $|\mathcal Z(x, y)|\le C|y|$.  
For the sake of brevity we use 
the representation \eqref{HS:eq} for compactly supported 
functions only, so that that the integral 
\eqref{HS:eq} is norm-convergent. 

Let us describe a 
convenient quasi-analytic extension of a function 
$f\in\plainC{n}_0(\R)$. 
For an arbitrary $r >0$ introduce the function 
\begin{equation}\label{ub:eq}
U(x, y) = 
\begin{cases}
1,\ |y|< \lu x\ru,\\[0.2cm]
0,\  |y|\ge \lu x\ru,
\end{cases}
\lu x\ru = \sqrt{x^2 + 1}. 
\end{equation}
The next proposition can be found in \cite[Ch. 2]{EBD}. 

\begin{prop}\label{ab:prop}
	Let $f\in \plainC{n}(\R), n\ge 2$. Then the function $f$ 
	has a quasi-analytic extension 
	$\tilde f = \tilde f(\ \cdot\ , \ \cdot)\in\plainC1(\R^2)$ such that 
	$\tilde f(x, y) = 0$ if $|y|>\lu x\ru$. Moreover, the derivative
	\begin{equation*}
		\mathcal Z(x, y) = \frac{\p}{\p \overline{z}}\tilde f(x, y),
	\end{equation*}
	satisfies the bound 
	\begin{equation}\label{ab:eq}
	|\mathcal Z(x, y)|
	\lesssim F(x)
	|y|^{n-1} U(x, y).
	\end{equation} 
	where 
	\begin{equation*}
		F(x) = \sum_{l=0}^{n} |f^{(l)}(x)|\lu x\ru^{-n+l}.
	\end{equation*}	
	The constant in \eqref{ab:eq} does not depend on $f$. 
\end{prop}

Denote
\begin{equation}\label{Nns:eq}
N_{n} (f) = \int F(x)\lu x\ru^{n-2} dx.
\end{equation}

\subsection{``Quasi-commutators"}
Let $\GH$ be a Hilbert space. Let 
$A, B : \GH \mapsto \GH$ be some bounded self-adjoint 
operators, and let $J: \GH\mapsto \GH$ be a bounded operator. 
Here we make some elementary observations about the ``quasi-commutator"
\begin{align*}
f(A)J-J f(B)
\end{align*}
with a smooth function $f$. 
We are interested in estimates in the normed ideal $\GS$ of 
compact operators with the norm $\|\ \cdot\ \|_\GS$. 

\begin{thm}\label{perturb:thm}
Let $A, B$ be two self-adjoint bounded operators, and let $J$ be a bounded operator. 	
Suppose that $f\in \plainC{n}_0(\R)$ with $n\ge 3$. Let $\GS$ be a normed ideal 
	of compact operators acting on $\GH$. Then   
\begin{align}\label{J0fun:eq}
\| f(A)J- J f(B)\|_{\GS}\lesssim  N_{n}(f) \|A J - J B\|_{\GS}.
		\end{align}		
		The constant in \eqref{J0fun:eq} 
		does not depend on $f$ or $A, B, J$. 
\end{thm}

\begin{proof}
Let us consider the function 
$f(t) = r_z(t) = (t-z)^{-1}$ 
with $\im z\not = 0$, and prove that 
	\begin{align}\label{J0:eq}
		\| r_z(A)J 
		- Jr_z(B)\|_{\GS}
		\le \frac{1}{|\im z|^2} \|A J - J B\|_{\GS}.
		\end{align}
By the resolvent identity 
\begin{align*}
r_z(A)J - J r_z(B) = 
- r_z(A)
(AJ - J B) r_z(B),
\end{align*}
we have 
\begin{equation*}
\|r_z(A)J - J r_z(B)\|_{\GS}\le \frac{1}{|\im z|^2}
\|A J - J B\|_{\GS},
\end{equation*}
whence \eqref{J0:eq}.

By formula \eqref{HS:eq} and Proposition \ref{ab:prop}, 
we have 
\begin{align}\label{funD:eq}
	\| f(A)J - &\ Jf(B)\|_{\GS}\notag\\[0.2cm]
	\lesssim &\  \underset{|y| < \lu x\ru}\iint F(x) |y|^{n-1}
\| r_{x+iy}(A)J - Jr_{x+iy}(B)\|_{\GS}
	dx dy\notag\\
	\lesssim &\ \|A J - J B\|_{\GS}\underset{|y| < \lu x\ru}\iint F(x) |y|^{n-3}dx dy,
	\end{align}
	where we have used \eqref{J0:eq}. 
The right-hand side is clearly estimated by $N_n(f) \|A J - J B\|_{\GS}$, as required. 
\end{proof}

 From now on, unless otherwise stated we always assume that 
$f\in \plainC{n}_0(\R)$ with some $n\ge 3$.  
We apply the above simple result to the operators of the form
\begin{align*}
\CD(A_1, P_1; f) J - J \CD(A_2, P_2; f), 
\end{align*}
involving two pairs of self-adjoint bounded operators $A_1, A_2$ and 
$P_1, P_2$, where
\begin{align*}
\CD(A, P; f) = Pf(PAP)P - Pf(A)P. 
\end{align*}
We do not consider the most general case, but concentrate on a very 
special one, which is used  
later in the proof of the main theorem.

As before, the constants in the estimates below 
depend neither on the function $f$ 
nor operators involved.
 
\begin{cor}\label{on:cor}
Let $A_1, A_2$ be bounded self-adjoint operators, and let 
$P_1, P_2$ be bounded self-adjoint 
operators such that $\|P_1\|$, $\|P_2\|\le 1$. 
Let $J$ be a bounded operator. 
Suppose that 
\begin{equation}\label{on_comm:eq}
[P_1, J] = [P_2, J] = 0,
\end{equation}
and
\begin{equation}\label{on_loc:eq}
(P_1 - P_2) J = 0.
\end{equation}
Then  
\begin{align}\label{on:eq}
\| \CD(A_1, P_1; f)J 
- &\ J\CD(A_2, P_2; f)\|_{\GS}\notag\\
\lesssim &\  N_{n}(f) 
\bigl[\|(A_1-A_2)J\|_{\GS}+ \|[A_2, J]\|_{\GS}\bigr],
\end{align}
and 
\begin{align}\label{on1:eq}
\| J\CD(A_1, P_1; f) 
- &\ J\CD(A_2, P_2; f)\|_{\GS}\notag\\
\lesssim &\  N_{n}(f) 
\bigl[\|(A_1-A_2)J\|_{\GS}+ \|[A_1, J]\|_{\GS} + \|[A_2, J]\|_{\GS}\bigr],
\end{align}
\end{cor}

\begin{proof} 
Observe first that \eqref{on1:eq} follows from \eqref{on:eq} 
due to the following equality:
\begin{align*}
J\CD(A_1, P_1; f) 
- &\ J\CD(A_2, P_2; f)\\ 
= &\ [J, \CD(A_1, P_1; f)] + 
\CD(A_1, P_1; f)J - J\CD(A_2, P_2; f). 
\end{align*}

Proof of \eqref{on:eq}. By \eqref{on_comm:eq},
\begin{align*}
\|[P_1 A_2 P_2, J]\|_{\GS}\le \|[A_2, J]\|_{\GS}.
\end{align*}
Also, by  
\eqref{J0fun:eq}, 
\begin{align*}
\|[  f(P_2A_2 P_2), J]\|_\GS 
\lesssim N_n(f) \|[A_2, J]\|_{\GS}.
\end{align*}
Thus, in view of \eqref{on_comm:eq} and \eqref{on_loc:eq},  
\begin{align*}
P_1 f(P_1A_1 P_1) P_1J - &\ JP_2 f(P_2A_2 P_2) P_2\\
= &\ P_1\bigl( f(P_1 A_1 P_1)J - J f(P_2 A_2 P_2)\bigr) P_1 \\
+ &\ P_1[ f(P_2A_2 P_2), J]P_2 - P_1[ f(P_2A_2 P_2), J]P_1
\end{align*}
By \eqref{J0fun:eq} and \eqref{on_comm:eq}, \eqref{on_loc:eq}, 
the first term on the right-hand side 
does not exceed
\begin{align*}
N_n(f) \|P_1 A_1 P_1J - &\ JP_1 A_2 P_2\|_\GS\\
\le &\ N_n(f) \bigl(\|P_1 A_1 P_1J - P_1 A_2 P_1J\|_\GS 
+ \|[P_1 A_2P_2, J]\|_{\GS}\bigr)\\
\le &\ N_n(f) \bigl(\|(A_1 - A_2)J\|_\GS 
+ \|[A_2, J]\|_{\GS}\bigr).
\end{align*}
The second and the third terms are bounded by 
$N_n(f)\|[A_2, J]\|_{\GS}$.
This leads to \eqref{on:eq}. 
 \end{proof}

\begin{cor}\label{CD:cor}
Let $A$ be a bounded self-adjoin operator, and let $P$ be an orthogonal 
projection. Let $J$ be an operator such that $[P, J] = 0$. Then
\begin{equation}\label{CD:eq}
\|J\CD(A, P; f)\|_\GS + \|\CD(A, P; f)J\|_\GS\lesssim N_n(f) 
\bigl(\|JPA(I-P)\|_{\GS} 
+ \|[A, J]\|_{\GS} \bigr).
\end{equation}
The constant in \eqref{CD:eq} is independent 
of $A, P, J$ or $f$.
\end{cor}

\begin{proof} It suffices to prove \eqref{CD:eq} 
for the operator $\CD(A, P; f)J$ only. 
Rewrite it:
\begin{align*}
\CD(A, P; f)J
= &\ Pf(PAP)PJ - Pf(A)JP\\
= &\ Pf(PAP)PJ - PJ f(A)P 
-P[f(A), J]P\\
= &\ P\bigl(f(PAP)PJ - PJ f(A)\bigr)P
-P[f(A), J]P.
\end{align*}
By \eqref{J0fun:eq}, the $\GS$-norm of the 
second term is bounded by $N_n(f)\|[A, J\|_\GS$. By 
\eqref{J0fun:eq}, 
the $\GS$-norm of the first term is bounded by
\begin{align*}
N_n(f)\|PAPJ - PJA\|_{\GS}\le 
N_n(f)\bigl(\|JPA(I-P)\|_{\GS} 
+ \|[A, J]\|_{\GS} \bigr).
\end{align*}
This completes the proof of \eqref{CD:eq}. 
\end{proof}

\subsection{Elementary estimates 
for pseudo-differential operators}
Now we apply  
Corollaries \ref{on:cor} and \ref{CD:cor} 
to pseudo-differential operators. 
As above we assume that  
$f\in \plainC{n}_0(\R)$ with some $n\ge 3$. 
The quantity $N_n(f)$ 
is defined in \eqref{Nns:eq}. 

Below we always assume that $\varphi\in\plainC\infty(\R^d)$ 
is a bounded function. 
Often we also assume that 
for some sets $\L$ and $\Pi$,
\begin{equation}\label{localp:eq}
\supp\varphi\cap \L = \supp\varphi\cap \Pi.
\end{equation}

\begin{lem}\label{9Dp:lem}   
Let $b, \tilde b$ be some real-valued symbols, and let 
$\L$ and $\Pi$ be some sets. 
Then 
\begin{equation}\label{D_trace_norm:eq}
\|\varphi D_\a(b, \L; f)\|_1 \lesssim 
N_n(f)\bigl(\| \chi_\L\varphi\op_\a(b)(I-\chi_\L)\|_1 
+ \|[\op_\a(b), \varphi]\|_1\bigr).
\end{equation}
Suppose that \eqref{localp:eq} is satisfied. Then 
\begin{align}\label{9Dp:eq}
\|\varphi D_\a(b, \L; f)
- &\ \varphi D_\a(\tilde b, \Pi; f)\|_1\notag\\[0.2cm]
\lesssim &\ N_n(f)
\bigl(\|[\varphi, \op_\a(b)]\|_1 
+ \|\varphi\op_\a(b-\tilde b)\|_1
+ \|\op_\a(b-\tilde b)\varphi\|_1\bigr).
\end{align}
If, in addition, $\Pi = \R^d$, then
\begin{equation}\label{9Dp_away:eq}
\|\varphi D_\a(b, \L; f)\|_1\lesssim 
N_n(f) \|[\varphi, \op_\a(b)]\|_1.
\end{equation} 
The constants in the above bounds do not depend on 
the function $f$, $\varphi$, sets $\L$, $\Pi$ or symbols $b$, $\tilde b$.
\end{lem}

 \begin{proof} 
The bound \eqref{D_trace_norm:eq} follows from Corollary 
\ref{CD:cor}  
with $P = \chi_\L, A = \op_\a(b), J = \varphi$.

The bound \eqref{9Dp_away:eq} 
follows from \eqref{9Dp:eq} used with $b = \tilde b$, 
since $D_\a(b, \R^d; f) = 0$.  
 
Proof of \eqref{9Dp:eq}. 
Use Corollary \ref{on:cor} with 
\[
A_1 = \op_\a(b), 
A_2 = \op_\a(\tilde b), P_1 = \chi_\L, P_2 = \chi_\Pi,\ 
J = \varphi. 
\] 
The condition \eqref{on_comm:eq} is trivially satified. By 
\eqref{localp:eq}, the condition \eqref{on_loc:eq} 
is also satisfied. Thus \eqref{9Dp:eq} follows from \eqref{on1:eq}.
\end{proof}

\section{ Estimates for Wiener-Hopf operators}\label{WH:sect}

In this section we collect some Schatten-von Neumann bounds for 
Wiener-Hopf operators with symbols satisfying some general conditions. 
Our main objective is to ensure the explicit dependence of the bounds on the symbols.

\subsection{Some basic bounds} 
  
To control the scaling properties of functions we introduce the following norms:
\begin{equation}\label{norm:eq}
\SN^{(n)}(\eta; \tau)
= \underset{0\le k\le n}
\max \ \underset{\bxi}
\sup\  \tau^{k}|\nabla_{\bxi}^k 
\eta(\bxi)|, n = 1, 2, \dots. 
\end{equation}
First we give some bounds in Schatten-von Neumann classes $\GS_q, q\in (0, 1]$,  
established in \cite{Sob1}, but adjusted for our purposes 
in the current paper. 

Unless specified otherwise, 
below each of the sets $\L, \Om\subset\R^d$ is 
a Lipschitz region. 
If $\L$ (or $\Om$) 
is a basic Lipschitz domain, i. e. $\L = \G(\Phi)$ with 
a globally Lipschitz function $\Phi$,  
then the constants 
in the estimates obtained below 
are uniform in $\L$   
in the sense that they depend only 
on the constant $M$ in the bound $M_\Phi \le M$.

Below we often use a test-function 
$\varphi\in\plainC\infty_0(\R^d)$  
such that 
\begin{equation}\label{supporth:eq}
\textup{support of the function $\varphi$ is contained in 
	$B(\bz, \ell)$,}
\end{equation} 
with some constant $\ell >0$ and some $\bz\in\R^d$. 
The bounds below are uniform in $\bz$, since by translation one can always 
assume that $\bz = \bold0$.
The constants in 
all estimates 
are independent of the 
 parameters $\a, \ell$ and $\tau$.

\begin{lem} \label{between:lem} 
Suppose that $\varphi$ satisfies \eqref{supporth:eq}, and 
that the support of $\eta\in\plainC\infty_0(\R^d)$ 
is contained in a ball 
of radius $\tau >0$. Let $q\in (0, 1]$ and 
\begin{equation}\label{rq:eq}
r  = r_q = [(d+1)q^{-1}] + 1.
\end{equation}
\begin{enumerate}
\item
Let $\L$ and $\Om$ be Lipschitz regions. 
If $\a\ell\tau\gtrsim 1$, then 
\begin{align}\label{comm_chi:eq} 
\|[\varphi\op_\a(\eta), \chi_\L]\|_q
+  \|[\varphi\op_\a(\eta), &\ \op_\a(\chi_\Om)]\|_q\notag\\[0.2cm]
\lesssim &\ (\a\ell\tau)^{\frac{d-1}{q}} \SN^{(r)}(\varphi; \ell)
\SN^{(r)}(\eta; \tau).
\end{align}
\item
Let $\L$ be a Lipschitz region,  
and let $\Om$ satisfy Condition \ref{domain:cond}. Suppose that 
$\a\ell\gtrsim 1$. 
Then 
\begin{equation}\label{comm_chi_chi:eq}
\|\chi_\L
\chi_{B(\bz, \ell)}
\op_\a(\chi_\Om)(I-\chi_\L)\|_q
\lesssim \bigl[(\a\ell)^{d-1} \log (\a\ell + 1)\bigr]^{\frac{1}{q}}.
\end{equation} 
\item
Let $\L$ and $\Om$ satisfy Condition \ref{domain:cond}. 
If $\a \ell\gtrsim 1$, then  
\begin{equation}\label{between:eq}
\| [\varphi, \op_\a(\chi_\Om)]\|_q
\lesssim (\a\ell)^{\frac{d-1}{q}} \SN^{(r)}(\varphi; \ell). 
\end{equation}
If $\a \tau\gtrsim 1$, then 
\begin{equation}\label{between1:eq}
\|[\op_\a(\eta), \chi_\L]\|_q
\lesssim (\a\tau)^{\frac{d-1}{q}} \SN^{(r)}(\eta; \tau).
\end{equation}
\end{enumerate}
If $\L$ (or $\Om$) is basic Lipschitz, then the relevant 
bounds are uniform in $\L$ (or $\Om$).
\end{lem}  

\begin{proof}
The bounds \eqref{comm_chi:eq} follow from 
\cite[Theorem 4.2 and Corollary 4.4]{Sob1}. 

For bounded $\Om$ the bound 
\eqref{comm_chi_chi:eq} is easily deduced 
from \cite[Theorem 4.6 and Corollary 4.7]{Sob1}. 
We omit the details. If $\Om^c = \R^d\setminus\Om$ is bounded, 
then 
\begin{align*}
\chi_\L\varphi\op_\a(\chi_\Om)(I-\chi_\L) 
= - \chi_\L \varphi\op_\a(\chi_{\Om^c}) (I-\chi_\L),
\end{align*}
and we can use \cite{Sob1} again. 

For bounded $\L$ and $\Om$ the bounds 
\eqref{between:eq} and \eqref{between1:eq} 
follow from \eqref{comm_chi:eq} by using a suitable partition of unity, or 
one can use the appropriate result from 
\cite[Corollary 4.4]{Sob1}. 
In the case of bounded complements $\L^c$ and 
$\Om^c$ we use the 
obvious identities 
\[
[\varphi, \op_\a(\chi_\Om)]
= -[\varphi, \op_\a(\chi_{\Om^c})], \  
\ [\op_\a(\eta), \chi_\L]
= -[\op_\a(\eta), \chi_{\L^c})],
\] 
and \cite[Corollary 4.4]{Sob1} again.
\end{proof}

\begin{lem}\label{comm:lem} 
Let 
$a=a(\bxi)$ be a symbol, and let 
$\ell, \a >0$ be some numbers. 
Then for any $m\ge d+1$ and any $s \ge s_0 >1$,
\begin{equation}\label{decouple:eq}
\|
 \chi_{B(\bz, \ell)}
\op_{\a}(a)\bigl(1- \chi_{B(\bz, s\ell)}\bigr)\|_1
\lesssim ((s-1)\a \ell)^{d-m} 
 \|\nabla^m a\|_{\plainL1}. 
\end{equation} 
The implicit constant does not 
depend on $\ell, \a$, $\varphi$ and $a$, but may depend on $s_0$. 
\end{lem}
 
\begin{proof}
Without loss of generality we may assume that $\bz =  \bold0$. 
The operator in \eqref{decouple:eq} is unitarily equivalent 
to 
\[ 
 \chi_{B(\bold0, 1)}
\op_1(\tilde a)\bigl(I - \chi_{B(\bold0, s)}\bigr),\ 
\tilde a(\bxi) = a(\bxi(\a\ell)^{-1}).
\]
Since the sets 
$B(\bold0, 1)$ and $\R^d\setminus B(\bold0, s)$ 
are separated by a 
positive distance $s-1$, 
it follows from \cite[Theorem 2.6]{Sob1} that 
\begin{equation*}
\|
\chi_{B(\bold0, 1)}
\op_1(\tilde a)\bigl(I - \chi_{B(\bold0, s)}\bigr)\|_1
\lesssim 
(s-1)^{d-m}\| \nabla^m \tilde a\|_{\plainL1},
\end{equation*}
for any $m\ge d+1$.
The $\plainL1$-norm on the right-hand side coincides with 
$(\a\ell)^{d-m}\|\nabla^m a\|_{\plainL1}$, which leads to \eqref{decouple:eq}.
\end{proof}

\subsection{Bounds for more general operators}

For methodological purposes it is also necessary to 
introduce more general pseudo-differential operators. 
For a function $p = p(\bx, \by, \bxi)$, 
which we call \textit{amplitude}, define the operator 
\begin{equation*}
\bigl(\op_\a^{\rm a}(p) u\bigr)(\bx)
= \frac{\a^{d}}{(2\pi)^{\frac{d}{2}}}
\iint e^{i\a\bxi\cdot(\bx-\by)} p(\bx, \by, \bxi) 
u(\by) d\by d\bxi,  u\in \plainS(\R^d).
\end{equation*} 
We need a very simple-looking bound for the trace norm of 
$\op_\a^{\rm a}(p)$ which we borrow from  \cite[Theorem 2.5]{Sob1}:

\begin{lem}\label{ampl:lem}
\begin{equation}\label{ampl1:eq}
\|\op_1^{\rm a}(p)\|_1
\lesssim \sum_{n, l = 0}^{d+1}
\iiint |\nabla_{\bx}^n\nabla_{\by}^l p(\bx, \by, \bxi)|d\bx d\by d\bxi,
\end{equation}
and 
\begin{equation}\label{ampl2:eq}
\|\op_\a^{\rm a}(p)\|_1
\lesssim \a^d \sum_{n, l = 0}^{d+1}
\iiint |\nabla_{\bx}^n\nabla_{\by}^l p(\bx, \by, \bxi)|d\bx d\by d\bxi.
\end{equation}
for any $\a>0$. 
The implicit constants in \eqref{ampl1:eq} and \eqref{ampl2:eq} 
do not depend on $\a$ or amplitude $p$.
\end{lem}

\begin{proof}
The bound \eqref{ampl1:eq} is a direct consequence of \cite[Theorem 2.5]{Sob1}. 
The bound \eqref{ampl2:eq} follows from \eqref{ampl1:eq} 
by rescaling $\bxi\mapsto \bxi\a^{-1}$.
\end{proof}

The above estimates are convenient for us 
because they do not contain any derivatives w.r.t. $\bxi$.  

\begin{lem} 
Let $a = a(\bxi)$, and let $\varphi = \varphi(\bx)$ satisfy 
\eqref{supporth:eq}. 
Then for any $\a>0$ and $\ell>0$, we have  
\begin{equation}\label{comm_new:eq}
\|[\op_\a(a), \varphi]\|_1\lesssim (\a\ell)^{d-1}
\SN^{(d+2)}( \varphi; \ell)
\bigl[
\|\nabla a\|_{\plainL1} + (\a\ell)^{1-m} \|\nabla^m a\|_{\plainL1}
\bigr],
\end{equation}
with an arbitrary $m\ge d+1$. 
The implicit constant in \eqref{comm_new:eq} does not depend on 
$\a, \ell$, $\varphi$ or $a$.
\end{lem}

\begin{proof} Without loss of generality assume that $\bz = 0$ and $\ell = 1$. 
Let $\tilde\varphi\in\plainC\infty_0(\R^d)$ be 
a function such that $\tilde\varphi(\bx) = 1$ for $|\bx|<2$. 
Denote 
\begin{equation*}
\tilde p(\bx, \by, \bxi) 
= a(\bxi)
\tilde \varphi(\bx)
\bigl(\varphi(\bx) - \varphi(\by) \bigr)
\tilde \varphi(\by).
\end{equation*}
Then by \eqref{decouple:eq}, 
\begin{equation}\label{besk:eq}
\| [\varphi, \op_\a(a)] - \op_\a^{\rm a}(\tilde p)\|_1
\lesssim \a^{d-m} \SN^{(d+1)}(\varphi; 1) \|\nabla^m a\|_{\plainL1}, 
\end{equation}
for any $m\ge d+1$. 
Further proof we conduct for $d=1$, although we do not replace 
$d$ by its value in the estimates below. 
The case of arbitrary $d$ is done in a similar way 
with obvious modifications. 
Integrate by parts in $\xi$, so that
\begin{equation*}
\op_\a^{\rm a}(\tilde p) = -\frac{1}{i\a} \op_\a^{\rm a} (p),\ 
p(x, y, \xi) = \frac{\p_{\xi}\tilde p(x, y, \xi)}{x-y}.
\end{equation*}
Rewrite 
\begin{align*}
p(x, y, \xi) =  a'(\xi) \tilde \varphi(x) \tilde \varphi(y)
\int_0^1 \varphi'\bigl(x+t(y-x)\bigr) dt .
\end{align*}
Thus by \eqref{ampl2:eq},
\begin{equation*}
\|\op_\a^{\rm a}(\tilde p)\|_1\lesssim 
\a^{d-1} \SN^{(d+1)}(\nabla \varphi; 1) \| \nabla a\|_{\plainL1},
\end{equation*}
where the implicit constant depends on $\tilde\varphi$. 
Together with \eqref{besk:eq} this gives \eqref{comm_new:eq}. 
\end{proof}

\subsection{\textit{Multi-scale} symbols} \label{mscale:subsect} 
The bounds above are very convenient as they 
contain easily computable quantities, such as integral norms 
of symbol's derivatives. We also need other types of bounds where 
the dependence on the symbol $a$ is less explicit, but still sufficient 
for our needs. Following 
\cite{LeSpSo_15}, we achieve this by placing ourselves in the context of 
\textit{multi-scale} symbols.

Let $v = v(\bxi)$ and $\tau = \tau(\bxi)$ 
be some continuous, positive functions on $\R^d$.  
Consider a symbol $a\in \plainC{\infty}(\R^d)$ satisfying the bounds
\begin{equation}\label{scales:eq}
|\nabla_{\bxi}^k a(\bxi)|\lesssim  
\tau(\bxi)^{-k} v(\bxi),\ k = 0, 1, 2, \dots,\quad \bxi\in\R^d.
\end{equation}
It is natural to call $\tau$ a \textit{scale (function)} 
and $v$ the \textit{amplitude (function)}. 
We always assume that $\|v\|_{\plainL\infty}\le 1$ and  
\begin{equation}\label{tauinf:eq}
\tau_{\textup{\tiny inf}} := \inf_{\bxi\in\R^d}\tau(\bxi)>0. 
\end{equation}
Introduce the notation 
\begin{equation}\label{Vsigma:eq}
V_{\s, \om}(v, \tau)  := \int \frac{v(\bxi)^\s}{\tau(\bxi)^\om}d\bxi, \ 
\s>0, \om\in\R.
\end{equation}
Apart from the continuity we 
need some extra conditions on the scale and the amplitude. 
First we assume that 
$\tau$ is globally Lipschitz, i.e., for some $\nu \in (0, 1)$,
\begin{equation}\label{Lip:eq} 
|\tau(\bxi) - \tau(\boldeta)| \le \nu |\bxi-\boldeta|,\ \ \bxi,\boldeta\in\R^d,
\end{equation} 
with some $\nu >0$. 
By adjusting the implicit constants in \eqref{scales:eq} 
one may choose for $\nu$ an arbitrary positive value. 
We assume that this value can be picked in such a way that 
the amplitude $v$ satisfies the relation 
\begin{equation}\label{w:eq}
\frac{v(\boldeta)}{v(\bxi)}\asymp 1,\ \boldeta\in B\bigl(\bxi, \tau(\bxi)\bigr).
\end{equation}
%
%
In the next result we establish some bounds that 
depend explicitly on the functional parameters $v$ and $\tau$. 

The following result follows from  
\cite[Lemma 3.4 and Theorem 3.5]{LeSpSo_15}.

\begin{prop}\label{Crelle:prop}
Let $a$ be a symbol satisfying \eqref{scales:eq} with 
some scaling function $\tau$ and some amplitude $v(\bxi)$ 
for which \eqref{Lip:eq}, \eqref{tauinf:eq} and 
\eqref{w:eq} hold. Suppose that $\L$ is a Lipschitz region, 
and that $\a\ell\tau_{\textup{\tiny inf}} \gtrsim 1$. Then 
\begin{align}\label{crelle1:eq}
\|\chi_\L  \chi_{B(\bz, \ell)}
\op_\a(a)(I-\chi_\L)\|_q^q
\lesssim (\a\ell)^{d-1}
V_{q,1}(v, \tau). 
\end{align}
If $\L$ is basic Lipschitz, then this bound is uniform in $\L$.

Suppose in addition that 

-- $\L$ satisfies Condition \ref{domain:cond},

-- the function $f$ satisfies Condition \ref{f:cond} with some $\g>0$, $R>0$  
and  $n=2$, and that $\a \tau_{\textup{\tiny inf}} \gtrsim 1$. 

Then 
for any $\s < \min\{1,\g\}$, we have 
\begin{align}\label{crelle2:eq}
\|D_\a(a, \L; f)\|_1\lesssim  \a^{d-1}\1 f\1_2 R^{\g-\s}
 V_{\s,1}(v, \tau). 
\end{align}
The implicit constants in 
\eqref{crelle1:eq} and \eqref{crelle2:eq} 
do not depend on $\a$, $f$ and $R$, 
but depend on the region $\L$, on 
the implicit constants 
in \eqref{scales:eq}, \eqref{w:eq}, and on the parameter $\nu$. 
\end{prop}

\section{Bounds involving the symbol $a_T$. 
Proof of Theorem \ref{entropy:thm}}\label{at:sect}

\subsection{Elementary bounds for the symbol $a_T$} 
Let $a_T=a_{T, \Om}$ 
be a symbol satisfying Condition \ref{at:cond}, 
with a \textit{bounded} $\Om$.  
In order to derive some integral bounds for $a_T$ 
we need to obtain estimates for the function $\rho(\bxi)$ 
(see \eqref{distance:eq}) 
in terms of the Lipschitz functions responsible for the local representation 
of $\p\Om$. 

Since the region $\Omega$ is bounded and has 
finitely many connected components, 
we can cover the boundary $S = \p\Om$ 
with finitely many open balls $\{D_j\}$ of equal radii $r\le 1$, 
centred at 
some $\bxi_j\in S$, such that in each of the balls $D_j$ 
the boundary $S$, with an appropriate choice of coordinates,
 is the graph of a Lipschitz function 
$\Psi_j$ on $\R^{d-1}$:
\begin{equation}\label{boundary:eq}
 S\cap D_j = \{\bxi\in\R^d: \xi_d = \Psi_j(\hat\bxi)\}\cap D_j. 
 \end{equation}
We may also assume that 
 \begin{equation}\label{telo:eq}
 \Om\cap D_j
 = \{\bxi\in\R^d: \xi_d > \Psi_j(\hat\bxi)\}\cap D_j.
 \end{equation} 
Let $\tilde D$ be an open subset of $\R^d\setminus S$, such that
\begin{equation}\label{coverrd:eq}
\R^d = (\cup_j D_j) \cup\tilde D,
\end{equation}  
It is clear that one can choose the balls $D_j$ so that 
\begin{align}
\begin{cases}\label{dist:eq}
\rho(\bxi)\asymp |\xi_d - \Psi_j(\hat\bxi)|,\ \bxi\in D_j,\\[0.2cm]
\rho(\bxi)\asymp \lu \bxi\ru,\ \bxi\in\tilde D.
\end{cases}
\end{align} 
It is natural to view $a_T$ as a multi-scale symbol 
(see Subsect. \ref{mscale:subsect} for the definition). Indeed, 
the bounds \eqref{ata:eq} and \eqref{nablat:eq} imply that  
\begin{align}\label{at:eq}
|\nabla^m a_T(\bxi)|\lesssim (T+\tilde\rho(\bxi))^{-m} 
\lu \bxi\ru^{-\b}, m = 0, 1, 2, \dots, 
\end{align}
so $a_T$ satisfies \eqref{scales:eq} with
\begin{align}\label{vtau:eq}
v(\bxi) = \lu\bxi\ru^{-\b},\ \ 
\tau(\bxi) = \frac{1}{2}(\tilde\rho(\bxi) + T).
\end{align}
Since $|\nabla\tilde \rho|= 1$ a.e., the thus defined scale $\tau$ satisfies 
\eqref{Lip:eq} with $\nu = 1/2$. 
Furthermore, the function $v$ satisfies \eqref{w:eq}. 
Note also that 
$\tau_{\textup{\tiny inf}}\asymp T$.

\begin{lem}\label{atprop:lem} 
Let $a_T = a_{T, \Om}$ be as in Condition \ref{at:cond}, and 
let $0 < T\lesssim 1$. Then 
for any $\d> d\b^{-1}$, 
\begin{equation}\label{ld:eq}
\||a_{T} - \chi_{\Om}|^\d\|_{\plainL1}\lesssim T, 
\end{equation} 
and for any $m\ge 1$,
\begin{equation}\label{nablanorm:eq}
\|\nabla^m a_{T}\|_{\plainL1}\lesssim T^{-m+1}.
\end{equation}
Furthermore, let $v$ and $\tau$ be as defined in 
\eqref{vtau:eq}. Then for any $\s > d\b^{-1}$, 
\begin{align}\label{vlog:eq}
V_{\s, 1}(v, \tau)
\lesssim\log\biggl(\frac{1}{T}+1\biggr), 
\end{align}
and
\begin{align}\label{vp:eq}
V_{\s, \om}(v, \tau)
\lesssim T^{-\om+1}, \ \forall \om > 1. 
\end{align} 
The implicit constants in \eqref{ld:eq}, \eqref{nablanorm:eq}, \eqref{vlog:eq} 
and \eqref{vp:eq} depend only on the constants in \eqref{at:cond}. 
\end{lem}

\begin{proof} 
Proof of \eqref{ld:eq}. 
We estimate separately the integrals over domains $D_j$ and $\tilde D$. 
By \eqref{ata:eq} and \eqref{dist:eq}, 
the integral over $\tilde D$ does not exceed  
\begin{align*}
\underset{\tilde D}\int \lu |\bxi| T^{-1}\ru^{-\d\b} 
d\bxi  \lesssim  T^d \int_0^\infty \lu s\ru^{-\d\b} s^{d-1} ds\lesssim T^{d}.
\end{align*}
In the same way, 
in view of \eqref{dist:eq}, the integral over $D_j$ is bounded 
by 
\begin{align*}
 \underset{D_j}\int 
\lu |\xi_d - \Psi_j(\hat\bxi)|T^{-1}\ru^{-\d\b}d\bxi
\lesssim \int_0^\infty \lu s T^{-1}\ru^{-\d\b} ds
\lesssim T.
\end{align*} 
The bound \eqref{nablanorm:eq} is proved in a similar way: 
using \eqref{dist:eq} and \eqref{nablat:eq} 
we conclude that 
\begin{align*}
\underset{\tilde D}\int 
|\nabla^m a_T(\bxi)|d\bxi
\lesssim 
\int  
\lu|\bxi|T^{-1}\ru^{-\b}  
d\bxi\lesssim T^d, 
\end{align*}
and 
\begin{align*}
\underset{D_j}\int 
|\nabla^m a_T(\bxi)|d\bxi
\lesssim &\  \underset{D_j}\int 
(T+|\xi_d - \Psi_j(\hat\bxi)|)^{-m} 
\lu|\xi_d - \Psi_j(\hat\bxi)|T^{-1}\ru^{-\b}  
d\bxi\\[0.2cm]
\lesssim &\  T^{1-m}\int_0^\infty (1+s)^{-m-\b} ds,
\end{align*}
as required.

Proof of \eqref{vlog:eq} and \eqref{vp:eq}. 
As above, we use the covering \eqref{coverrd:eq} and 
estimate the integrals over $D_j, \tilde D$ separately, so that 
\begin{align*}
\underset{D_j}\int (T+\tilde\rho(\bxi))^{-\om} \lu \bxi\ru^{- \s\b} d\bxi
\lesssim &\ 
\underset{D_j}\int (T+|\xi_d-\Psi_j(\hat\bxi)|)^{-\om} d\bxi \\[0.2cm]
\lesssim &\ \int_0^1 \frac{1}{(T+s)^\om} ds\lesssim 
\begin{cases}
\log\biggl(\frac{1}{T}+1\biggr), \om = 1,\\
T^{1-\om},\ \om > 1,
\end{cases}
\end{align*}
and 
\begin{align*}
\underset{\tilde D}\int (T+\tilde\rho(\bxi))^{-\om} 
\lu\bxi\ru^{-\s\b} d\bxi
\lesssim 
\int \lu\bxi\ru^{-\s\b} d\bxi\lesssim 1.
\end{align*}
This proves \eqref{vlog:eq} and \eqref{vp:eq}.
\end{proof}

\subsection{Bounds for pseudo-differential operators with $a_T$}
     
Without delay we infer the following useful consequence of the above bounds. 
  
\begin{lem}\label{scalesT:lem}
Let  $\Om$ be a bounded Lipschitz region. 
Let $\a, \ell>0$ be some numbers and let $0<T\lesssim 1$. 
If $\varphi$ satisfies \eqref{supporth:eq}, then for arbitrary 
$m\ge d+1$, and any $s >1$,
\begin{equation}\label{scales_apart:eq}
\| \chi_{B(\bz, \ell)} 
\op_\a(a_{T}) 
\bigl(I- \chi_{B(\bz, s\ell)}\bigr)\|_1
\lesssim (\a\ell)^{d-1} (\a\ell T)^{1-m}, 
\end{equation}
and 
\begin{align}\label{comm_trace:eq}
  \|[\varphi, \op_\a(a_{T})]\|_1
\lesssim (\a\ell)^{d-1} 
\SN^{(d+2)}(\varphi; \ell)\bigl(1+ (\a\ell T)^{1-m}\bigr).
\end{align}	  
The implicit constants in \eqref{scales_apart:eq} and 
\eqref{comm_trace:eq} are independent of $\a, \ell, T$ or $\varphi$. 
\end{lem}

\begin{proof} 
The bound \eqref{scales_apart:eq} follows from \eqref{decouple:eq} 
and \eqref{nablanorm:eq}. 
The bound \eqref{comm_trace:eq} follows from \eqref{comm_new:eq} 
and \eqref{nablanorm:eq}. 
 \end{proof}
   
Let us use Proposition \ref{Crelle:prop} for the symbol $a_T$.

\begin{prop}\label{logT:prop} 
Suppose that $\L$ is a Lipschitz region, and that 
$\Om$ is a bounded Lipschitz region. 
If $\a\ell T \gtrsim 1$, then for any $q\in (d\b^{-1}, 1]$ we have 
\begin{equation}\label{logT:eq}
\|\chi_\L
\chi_{B(\bz, \ell)}
\op_\a(a_{T})(I-\chi_\L)\|_q^q\lesssim 
(\a\ell)^{d-1} \log \biggl(\frac{1}{T}+1\biggr).
\end{equation}
If $\L$ is basic Lipschitz, then this bound is uniform in $\L$. 
Suppose in addition that 

-- $\L$ satisfies Condition \ref{domain:cond}, 

-- $f$ satisfies Condition \ref{f:cond} with some $\g >0$, $R>0$ and 
$n=2$, and 
that $\b > \max\{d\g^{-1}, d\}$. 

If $\a T\gtrsim 1$, then  
for any $\s \in (d\b^{-1}, \g)$, $\s < 1$:
\begin{align}\label{logda:eq}
\|D_\a(a_T, \L; f)\|_1\lesssim \a^{d-1}
\1 f\1_2 R^{\g-\s} 
\log \biggl(\frac{1}{T}+1\biggr).
\end{align}
The implicit constants in \eqref{logT:eq} and 
\eqref{logda:eq}
are independent of $\a, \ell, T$ and $f$.

\end{prop}

\begin{proof} 
Since $\tau_{\textup{\tiny inf}}\asymp T$, we have 
$\a\ell\tau_{\textup{\tiny inf}}\asymp \a\ell T\gtrsim 1$. So 
the bounds \eqref{crelle1:eq}
 and \eqref{vlog:eq} 
lead to \eqref{logT:eq}.
Under the condition $\a\tau_{\textup{\tiny inf}}\asymp \a T\gtrsim 1$
the bounds \eqref{crelle2:eq} and \eqref{vlog:eq} lead 
to \eqref{logda:eq}.
\end{proof}

\subsection{Lattice norm bounds 
for pseudo-differential operators}

For a function $u:\R^d\mapsto \mathbb C$ denote
\begin{equation*}
\2 u\2_q = \biggl[\sum_{\bn\in\mathbb Z^d} 
\biggl(\underset{\bn + [0, 1)^d}\int |u(\bx)|^2 d\bx\biggr)^{\frac{q}{2}}
\biggr]^{\frac{1}{q}}.
\end{equation*}
If $q\ge 1$, this formula defines a norm, sometimes called 
\textit{a lattice norm}, 
and if $q <1$, then -- quasi-norm, called \textit{lattice quasi-norm}. 
The following result is well-known, see 
\cite[Theorem 11.1]{BS_U}, \cite[Section 5.8]{BKS}, 
and for $q\in [1, 2)$ -- 
\cite[Theorem 4.5]{Simon}. 

\begin{prop}\label{BS:prop}
If $\2 w\2_q, \2 b\2_q<\infty$ for some 
$q\in (0, 2]$, then 
\begin{equation*}
\| w \op_1(b)\|_q\lesssim  \2 w\2_q\  \2 b\2_q.
\end{equation*}
\end{prop}

\begin{cor}\label{cyl:cor} 
Let $\a, \ell>0$ and $0<T\lesssim 1$ 
be such that $\a\ell\gtrsim 1$ and 
$\a T\ell \lesssim 1$. 
Let $\Om\subset\R^d$ and $a_T = a_{T, \Om}$ be as in Condition 
\ref{at:cond}. 
Then for any $q\in (d\b^{-1}, 1]$ 
\begin{equation}\label{cyl:eq}
\| \chi_{B(\bz, \ell)} 
\op_\a\bigl(a_{T} - \chi_{\Om}\bigr)\|_q^q
\lesssim (\a \ell)^{d-1},
\end{equation}
with an implicit constant independent of $\bz\in\R^d$ or $\a, \ell, T$. 
\end{cor} 

\begin{proof}
Let $\{\phi_j\}, \tilde\phi$ be a partition of unity 
subordinate to the cover \eqref{coverrd:eq}. 
Denote for brevity $a = a_{T}$.  
Estimate separately the operators 
\[
Z_j = \chi_{B(\bz, \ell)} \op_\a\bigl(\phi_j(a - \chi_\Om)\bigr) \ \ 
\textup{and}\ \  
\tilde Z = \chi_{B(\bz, \ell)} \op_\a\bigl(\tilde\phi(a - \chi_\Om)\bigr).
\]
Let $\Psi = \Psi_j\in\textup{Lip}(\R^{d-1})$ 
be a function describing the surface $S$ inside $D_j$, 
see \eqref{boundary:eq}. Recall that we always assume that 
$\|\nabla\Psi\|_{\plainL\infty}\lesssim 1$.
Without loss of generality assume that $\bz = 0$. By rescaling, the operator 
$Z_j$ is unitarily equivalent to 
\begin{equation*}
Z_j' = \chi_{B(\bold0, 1)} \op_1(b),\ b(\bxi) = 
\bigl[a\bigl(\bxi (\a\ell)^{-1}\bigr) - \chi_\Om\bigl(\bxi(\a\ell)^{-1}\bigr)\bigr]
\phi_j\bigl(\bxi(\a\ell)^{-1}\bigr),
\end{equation*} 
By Proposition \ref{BS:prop}, 
\begin{equation*}
\|  \chi_{B(\bold0, 1)} \op_1(b)\|_q
\lesssim  \2  \chi_{B(\bold0, 1)}\2_q \2 b\2_q.
\end{equation*} 
It is clear that $\2  \chi_{B(\bold0, 1)}\2_q<\infty$ and 
it is independent of any parameters. Let us estimate $\2 b\2_q$. 
By \eqref{ata:eq},
\begin{equation*}
|b(\bxi)|\lesssim \biggl\lu 
\frac{|\xi_d - \tilde \Psi(\hat\bxi)|}{\a\ell T}\biggr\ru^{-\b}
\chi_{B(\bold0, \a \ell)}(\bxi),\ 
\tilde\Psi(\bxi) = \a\ell \Psi\bigl(\hat\bxi(\a\ell)^{-1}\bigr).
\end{equation*}
Define the sets 
\begin{equation*}
\CO_s = \{\bxi\in B(\bold0, \a\ell): 
s\le |\xi_d - \tilde\Psi(\hat\bxi)|< s+1\}, 
s = 0, 1, \dots.
\end{equation*}
Since $\|\nabla\tilde\Psi\|_{\plainL\infty} 
= \|\nabla\Psi\|_{\plainL\infty}\lesssim 1$, the number 
\begin{equation*}
\#\{\bn\in\mathbb Z^d: \bn + [0, 1)^d
\cap  \CO_s\not = \varnothing\}
\end{equation*}
does not exceed $(\a\ell)^{d-1}$, uniformly in $s$.
As a result,
\begin{equation*}
\2 b\chi_{\CO_s}\2_q\lesssim \lu s \ru^{-\b}(\a\ell)^{\frac{d-1}{q}}, 
\end{equation*}
where we have used the property $\a\ell T\lesssim 1$. 
Thus, by the $q$-triangle inequality, we have
\begin{align*}
\2 b\2_q^q\le \sum_s \2 b\chi_{\CO_s}\2_q^q
\lesssim (\a\ell)^{d-1}\sum_s \lu s\ru^{-\b q}\lesssim (\a\ell)^{d-1}.
\end{align*}
Let us now consider the operator 
$\tilde Z = \chi_{B(\bold0, \ell)}\op_\a(\tilde b)$, 
$\tilde b = \tilde\phi(a - \chi_\Om)$. 
 By rescaling, 
the operator $\tilde Z$ is unitarily equivalent to 
\begin{equation*}
\tilde Z' = \chi_{B(\bold0, \a\ell)}\op_1(\tilde b),  
\end{equation*}
By virtue of \eqref{dist:eq}, $\rho(\bxi)\gtrsim |\bxi|$, $\bxi\in\tilde D$, 
and hence \eqref{ata:eq} implies that  
\begin{equation*}
\2\tilde b\2_q^q\lesssim 
\sum_{\bn\in\Z^d} \lu \bn T^{-1}\ru^{-\b q}
\lesssim \int_0^\infty \biggl\lu \frac{s}{T}\biggr\ru^{-\b q} s^{d-1} ds
= T^d\int_0^\infty \lu s\ru^{-\b q} s^{d-1}ds \lesssim T^d. 
\end{equation*}
Furthermore, since $\a\ell\gtrsim 1$, we have 
\begin{equation*}
\2 \chi_{B(\bold0, \a\ell)}\2_q^q\lesssim (\a\ell)^d.
\end{equation*}
By Proposition \ref{BS:prop},
\begin{equation*}
\| \tilde Z\|_q^q\lesssim (\a\ell)^d  T^d
\lesssim (\a\ell T)^{d} \lesssim 1.
\end{equation*}
Collecting the contributions from all $D_j$'s 
and $\tilde D$, we get the bound \eqref{cyl:eq}, as claimed.
\end{proof}
 
\begin{cor} Let 
$\varphi$ satisfy \eqref{supporth:eq}, and 
let $\Om$ be a bounded Lipschitz region.  
Suppose that $\a\ell\gtrsim 1$. Then for any $0<T\lesssim 1$ we have 
\begin{equation}\label{between2:eq}
\|[\varphi, \op_\a(a_{T})]\|_1\lesssim \SN^{(d+2)}(\varphi; \ell) 
(\a\ell)^{d-1}.
\end{equation}
Moreover, for any Lipschitz region $\L$, and any $q\in (d\b^{-1}, 1]$ we have  
\begin{equation}\label{logalT:eq}
\|\chi_\L 
\chi_{B(\bz, \ell)}
\op_\a(a_T)(I-\chi_\L)\|_q^q
\lesssim  
(\a\ell)^{d-1}
\log \biggl(\min\biggl\{\a\ell, \frac{1}{T}\biggr\}+1\biggr).
\end{equation}
%
%
The constants in \eqref{between2:eq} and \eqref{logalT:eq} 
are independent of $\a, \ell, T$ or $\varphi$.
\end{cor}

\begin{proof}
The bound \eqref{between2:eq}
holds for $\a\ell T\gtrsim 1$, due to \eqref{comm_trace:eq}. 
For $\a\ell T\lesssim 1$ we use \eqref{cyl:eq} and \eqref{between:eq} to get
\begin{align*}
\|[\varphi, \op_\a(a_{T})]\|_1\le &\ 
\|[\varphi, \op_\a(\chi_{\Om})]\|_1
+ 2\|\varphi\op_\a(a_{T} - \chi_{\Om})\|_1\\[0.2cm]
\lesssim &\ \SN^{(d+2)}(\varphi; \ell) (\a\ell)^{d-1}.
\end{align*}
Thus \eqref{between2:eq} is proved. 

Proof of \eqref{logalT:eq}. If $\a\ell T\gtrsim 1$, 
then \eqref{logalT:eq} follows directly from 
\eqref{logT:eq}. For $\a \ell T\lesssim 1$ estimate using the $q$-triangle inequality 
\eqref{qtriangle:eq}:
\begin{align*}
\|\chi_\L \chi_{B(\bz, \ell)}
\op_\a(a_T)(I-\chi_\L)\|_q^q
\le &\ \|\chi_{B(\bz, \ell)}
\op_\a(a_T-\chi_\Om)\|_q^q \\[0.2cm]
+ &\ \|\chi_\L 
\chi_{B(\bz, \ell)}
\op_\a(\chi_\Om)(I-\chi_\L)\|_q^q.
\end{align*}
By \eqref{cyl:eq} and \eqref{comm_chi_chi:eq}, 
the right-hand side satisfies the required bound. 
\end{proof}

\subsection{Proof of Theorem \ref{entropy:thm}} 
Recall that by Remark \ref{boundom:rem} 
it suffices to prove Theorem \ref{entropy:thm} 
for a bounded $\Om$.  In this case 
the bound \eqref{entropy:eq} for $\a T\gtrsim 1$ 
is already proved in Proposition \ref{logT:prop}. 

Suppose that $\a T\lesssim 1$. 
It immediately follows from Proposition \ref{Szego1:prop} 
with $P = \chi_{\L}, A = \op_{\a}(a_T)$, that 
\begin{equation*}
\| D(a_{T}, \L; f)\|_1\lesssim \1 f\1_2 R^{\g-q} 
\|\chi_\L \op_\a(a_{T})(I -\chi_\L)\|_q^q,
\end{equation*}
for any $q   < \min\{1, \g\}$. 
Let   $B(\bold0, R)$ be a ball such that 
either $\L\subset B(\bold0, R)$ or 
$\R^d\setminus\L\subset B(\bold0, R)$. 
Thus the $\GS_q$-norm on the right-hand side is estimated either by
\begin{align*}
\| \chi_{B(\bold0, R)}
\chi_\L \op_\a(a_T)(I-\chi_\L)\|_q^q
\end{align*}
or by 
\begin{equation*}
\|[\chi_\L\op_\a(a_T) \chi_{B(\bold0, R)}
(I-\chi_\L)\|_q^q.
\end{equation*}
For $q\in (d\b^{-1}, \g)$, $q\le 1$ both 
(quasi-)norms are bounded 
as in \eqref{logalT:eq}. This leads to the proclaimed bound \eqref{entropy:eq}. 
\qed 
 
\begin{rem}
In the proof of the main Theorem \ref{main:thm} we also need 
a version of the bound \eqref{entropy:eq} for smooth functions $f$. 
Suppose that $g\in\plainC{2}_0(-r, r), r>0$. 
Then,arguing as in the proof above but using Proposition 
\ref{smoothP:prop} 
instead of Proposition \ref{Szego1:prop}, 
we obtain the bound 
\begin{equation*}
\| D(a_{T}, \L; g)\|_1\lesssim \| g\|_{\plainC{2}} 
\|\chi_\L \op_\a(a_{T})(I -\chi_\L)\|_q^q,
\end{equation*}
for any $q < 1$, with an implicit constant independent of 
$g$, but dependent on the 
number $r$.  As in the proof above, this bound in combination with \eqref{logalT:eq} leads to the estimate 
\begin{align}\label{smooth2:eq}
\| D(a_{T}, \L; g)\|_1\lesssim 
\| g\|_{\plainC{2}} \ 
\a^{d-1}
\log \biggl(\min\biggl\{\a, \frac{1}{T}\biggr\}+1\biggr). 
\end{align}
\end{rem}

\section{Asymptotics for discontinuous symbols}
\label{discont:sect}
  
In the proof of the main theorem we use two types of asymptotics 
for $D_\a(\chi_\Om, \L; g_p)$ for polynomials $g_p(t) = t^p$, 
$p = 1, 2, \dots$, established in \cite{Sob} and \cite{Sob2}. 
Recall that the integrals $\GV_1(b; \p\L, \p\Om)$ 
and $U(f)$ are defined in \eqref{W1:eq} 
and \eqref{U:eq} respectively. 

In the case $\a T \lesssim 1$ we need the following fact, 
see \cite[Lemma 4.3]{Sob2}.

\begin{prop}\label{IETO:prop}
Let $\L$ and $\Om$ be regions in $\R^d$ satisfying 
Condition \ref{domain:cond}, and let $\Om$ be bounded. 
Assume also that $\Om$ is piece-wise $\plainC3$ and $\L$ is piece-wise 
$\plainC1$. 
Then for any function $\varphi\in\plainC\infty_0(\R^d)$, we have  
\begin{equation*}
\underset{\a\to\infty}
\lim\ \frac{1}{\a^{d-1}\log\a} \tr \bigl(\varphi D_\a(\chi_\Om, \L; g_p)\bigr)
= U(g_p)\GV_1(\varphi; \p\L, \p\Om).
\end{equation*}
\end{prop}
  
Under the condition $\a T\gtrsim 1$ we appeal to more subtle results   
from \cite{Sob}. These are not stated in 
\cite{Sob} exactly in the required form, hence  we 
need to do some extra work.  
Let $\L = \G(\Phi)$ be a basic $\plainC1$-domain.
We need to control the modulus of continuity of $\nabla\Phi$, hence 
we assume that 
\begin{equation}\label{Phieps:eq}
\sup_{\hat\bx, \hat\by: |\hat\bx-\hat\by|<r}
|\nabla\Phi(\hat\bx)-\nabla\Phi(\hat\by)| \le \varepsilon(r), 
\end{equation}
for some non-negative 
function $\varepsilon$ such that $\varepsilon(r)\to 0 $ 
as $r\to 0$. 

\begin{rem}\label{dom_uni:rem}
The asymptotic formula stated in the next Proposition 
is uniform in the basic domain $\L$ in the sense that the 
convergence is uniform in all functions $\Phi$  satisfying 
 the bound $M_\Phi\le M$ and \eqref{Phieps:eq} with some 
 constant $M$ 
 and some function $\varepsilon=\varepsilon(r)$.  
In particular, the domain $\L$ is allowed to depend on the large parameter $\a$ 
as long as the bounds $M_\Phi\le M$ 
and \eqref{Phieps:eq} hold with 
some $\a$-independent $M$ and $\varepsilon(r)$. 
\end{rem}

\begin{prop}\label{discont:prop} 
Let $\L$ be a basic 
$\plainC1$-domain, and let $\Om$ be a bounded $\plainC3$-region. 
Let $g_p(t) = t^p, t\in\R, p = 1, 2, \dots$, 
and let $\varphi = \varphi(\bx)$, $\eta = \eta(\bxi)$ 
be functions such that   
$\varphi\in\plainC\infty_0(B(\bz, R))$ and 
$\eta\in\plainC\infty_0(B(\bmu, R_1))$ 
with some $\bz, \bmu\in\R^d$, and some fixed $R, R_1>0$.   
Then 
\begin{align}\label{AMS:eq}
\lim_{\a\to\infty}\biggl[\frac{1}{\a^{d-1}\log\a}
\tr \bigl(\varphi \op_\a(\eta) D_\a(\chi_\Om, \L;  g_p)\bigr) 
- U(g_p)\GV_1(\varphi\eta; \p\L, \p \Om)\biggr] = 0.
\end{align} 
The convergence is uniform
\begin{enumerate}
\item
in the domain $\L$ in the sense specified in Remark \ref{dom_uni:rem}, 
\item 
in the functions $h$ and $\eta$ in the sense 
that it is uniform in the 
functions $\varphi$, $\eta$ satisfying the bounds
$\SN^{(k)}(\varphi; R), \SN^{(k)}(\eta; R_1)\lesssim 1$ for 
all $k = 0, 1, 2, \dots$, with some fixed constants.
\end{enumerate}
\end{prop}

This proposition follows from \cite[Theorem 11.1]{Sob}. 
%

The next Proposition makes a statement similar to \eqref{AMS:eq}, but 
uniform in the radius $R$. 

\begin{lem} \label{onebasic:lem}
Suppose that $\L$ and $\Om$ are as in Proposition 
\ref{discont:prop}.
Let $\varphi\in\plainC\infty_0(B(\bz, R))$ with some $R\le 1$. 
Then 
\begin{align}\label{onebasic:eq}
\lim_{\a R\to\infty} R^{1-d}\biggl[\frac{1}{\a^{d-1}\log(\a R)}
\tr \bigl(\varphi D_\a(\chi_\Om, \L;  g_p)\bigr) 
- U(g_p)\GV_1(\varphi ; \p\L, \p \Om)\biggr] = 0.
\end{align} 
The convergence is uniform in the domain $\L$ and the function $\varphi$
as specified in Proposition \ref{discont:prop}. 
\end{lem}

\begin{proof} 
First we note that \eqref{AMS:eq} with an arbitrary 
$\eta\in\plainC\infty_0(\R^d)$, implies \eqref{AMS:eq} with $\eta \equiv 1$. 
Indeed, Let $\eta\in\plainC\infty_0$ be such that $\eta\chi_\Om = \chi_\Om$. 
Write:
\begin{equation*}
\varphi \op_\a(\eta) \chi_\L \op_\a(\chi_\Om) 
= [\varphi \op_\a(\eta), \chi_\L] \op_\a(\chi_\Om)
+ \varphi \chi_\L \op_\a(\chi_\Om),
\end{equation*} 
so that 
\begin{align*}
\varphi\op_\a(\eta) \bigl(\chi_\L \op_\a(\chi_\Om)\chi_\L\bigr)^p 
- \varphi&\ \bigl(\chi_\L \op_\a(\chi_\Om)\chi_\L\bigr)^p\\[0.2cm]
= &\ [\varphi \op_\a(\eta), \chi_\L] \op_\a(\chi_\Om)
\bigl(\chi_\L \op_\a(\chi_\Om)\chi_\L\bigr)^{p-1},
\end{align*}
for any $p = 1, 2, \dots$. Therefore
\begin{equation*}
\| \varphi \op_\a(\eta) D_\a(\chi_\Om, \L;  g_p) 
- \varphi D_\a(\chi_\Om, \L;  g_p) \|_1
\le 2 \|[\varphi \op_\a(\eta), \chi_\L]\|_1.
\end{equation*}
By \eqref{comm_chi:eq}, the right-hand side  
does not exceed  
$\SN^{(d+2)}(\varphi; R)(\a R)^{d-1}$, uniformly in the domain $\L$. Therefore 
\eqref{AMS:eq} leads to \eqref{AMS:eq} with $\eta\equiv 1$, as claimed. 

By rescaling, the operator in \eqref{onebasic:eq} 
is unitarily equivalent to 
\begin{equation*}
\tilde \varphi D_\nu(\chi_\Om, \tilde\L; g_p),\ \nu = \a R\to \infty,  
\end{equation*}
with 
\begin{equation*}
\tilde \varphi(\bx) = \varphi(R\bx),\ \tilde \L = \G(\tilde\Phi),\ 
\tilde\Phi(\hat\bx) = R^{-1} \Phi(R\hat\bx).
\end{equation*}
Clearly, 
$\tilde \varphi\in\plainC\infty_0( R^{-1}\bz, 1)$, 
and $\SN^{(k)}(\tilde \varphi; 1)\lesssim 1$ 
for all $k= 0, 1, 2, \dots$. 
Furthermore, $M_{\tilde\Phi} = M_{\Phi}$, and,  
because of the restriction $R\le 1$ the functions $\nabla\Phi$ and 
$\nabla\tilde\Phi$ satisfy \eqref{Phieps:eq} with the same modulus 
of continuity $\varepsilon(r)$.
Thus one can use formula \eqref{AMS:eq} with $\eta\equiv 1$:
\begin{align*} 
\lim_{\nu\to\infty}\biggl[\frac{1}{\nu^{d-1}\log\nu}
\tr \bigl(\tilde \varphi D_\nu(\chi_\Om, \tilde\L;  g_p)\bigr) 
- U(g_p)\GV_1(\tilde \varphi ; \p\tilde\L, \p\Om)\biggr] = 0.
\end{align*} 
Observing that  
\begin{equation*}
\GV_1(\tilde \varphi; \p\tilde\L, \p \Om) 
= R^{1-d}\GV_1( \varphi; \p\L, \p \Om),
\end{equation*} 
we get \eqref{onebasic:eq}, as claimed. 
\end{proof}

\section{Proof of Theorem \ref{main:thm} for $\a T \lesssim 1$}
\label{proof1:sect}

In this section we begin the proof of Theorem \ref{main:thm}. 

\subsection{Localization estimates for the operator $D_\a(a, \L; f)$} 
Using Lemma \ref{9Dp:lem}, here we convert the bounds obtained previously 
in Sect. \ref{at:sect} into appropriate bounds for the operator 
\eqref{Dalpha:eq}. In Lemmas \ref{9D:lem} - \ref{ext:lem} 
we assume that $f\in \plainC{n}_0(\R)$ with $n\ge 3$. 
Unless otherwise stated, 
the constants in the estimates below do not depend on the {\color{red}function}
$f$ or parameters $\a, \ell, T$.

\begin{lem}\label{9D:lem} 
Let the symbol $a$ be either $a_T=a_{T, \Om}$ or $\chi_\Om$ with a 
bounded Lipschitz region $\Om$. 
Suppose that  
the  sets $\L$, $\Pi$ satisfy  
\begin{align}\label{localp1:eq}
B(\bz, 2\ell)\cap \L = B(\bz, 2\ell)\cap \Pi.
\end{align}
Suppose also that $\a\ell \gtrsim 1$, 
$0 < T\lesssim 1$. Then 
\begin{equation}\label{9D:eq}
\|
\chi_{B(\bz, \ell)} 
\bigl(D_\a(a, \L; f) - 
D_\a(a, \Pi; f)\bigr)\|_1\lesssim  N_n(f) 
(\a\ell)^{d-1}.
\end{equation}
If $\Pi = \R^d$, then
\begin{equation}\label{9D_away:eq}
\|
\chi_{B(\bz, \ell)} 
D_\a(a, \L; f)\|_1\lesssim N_n(f)
(\a\ell)^{d-1}.
\end{equation}
The constants in \eqref{9D:eq} and \eqref{9D_away:eq} 
do not depend on the sets $\L$ and $\Pi$. 

If, in addition, 
$\Pi$ is a Lipschitz region, then 
\begin{equation}\label{log_bound:eq}
\|\chi_{B(\bz, \ell)} D_\a(a_{T}, \L; f)\|_1
\lesssim N_n(f) 
(\a\ell)^{d-1}\log 
\biggl(\min\biggl\{
\a\ell, \frac{1}{T}
\biggr\}+1\biggr).
\end{equation}
If $\Pi$ is a basic Lipschitz domain, than the constant in 
\eqref{log_bound:eq} is uniform in $\Pi$.
\end{lem}

\begin{proof}  
The bound \eqref{9D_away:eq} is a direct consequence of \eqref{9D:eq}, since 
$D_\a(a, \R^d; f) = 0$.

To prove \eqref{9D:eq} 
let $\varphi\in\plainC\infty_0$ be a function such that 
$\varphi(\bx) = 1$, $\bx\in B(\bz, \ell)$, $\varphi(\bx) = 0$ 
for $\bx\notin B(\bz, 2\ell)$, and 
$\ell^m|\nabla^m\varphi|\lesssim 1$ for all $m = 1, 2, \dots$.  
Since $\chi_{B(\bz, \ell)}\varphi = \varphi$, 
in view of \eqref{localp1:eq}, the relation \eqref{localp:eq} is satisfied, and hence    
we can use Lemma \ref{9Dp:lem} with $\tilde b = b = a$.  
It follows from 
\eqref{between:eq} or \eqref{between2:eq} that
\begin{align*}
\|[\varphi, \op_\a(a)]\|_1\lesssim &\ 
(\a\ell)^{d-1}. 
\end{align*} 
Now \eqref{9Dp:eq} leads to \eqref{9D:eq}. 

Proof of \eqref{log_bound:eq}. 
We use the same function $\varphi$ as above. 
By \eqref{9D:eq} we may assume that $\L = \Pi$. 
Due to \eqref{D_trace_norm:eq}, the left-hand side 
of \eqref{log_bound:eq} does not exceed
\begin{equation*}
N_n(f) \bigl(\|\varphi\chi_\Pi \op_\a(a_{T})(I-\chi_\Pi)\|_1 
+ \|[\varphi, \op_\a(a_T)]\|_1\bigr).
\end{equation*}
It remains to apply \eqref{between2:eq} and 
\eqref{logalT:eq}.
\end{proof}

\begin{lem}\label{jump:lem}
Let the sets $\L$ and $\Pi$ 
satisfy \eqref{localp1:eq}.  
Suppose that $\Om$ is a bounded Lipschitz region. 
Let 
$\a\ell\gtrsim 1$ and $\a \ell T \lesssim 1$. Then 
\begin{align}\label{jump:eq}
\|
\chi_{B(\bz, \ell)}
\bigl(D_\a(a_{T}, \L; f) 
-  
D_\a(\chi_{\Om}, \Pi; f)\bigr)\|_1
\lesssim  N_n(f)
(\a\ell)^{d-1}. 
\end{align}
The constant in \eqref{jump:eq} is independent of 
the sets $\L$ and $\Pi$.
\end{lem}

\begin{proof} 
We use Lemma \ref{9Dp:lem} with 
$b = a_{T}, \tilde b = \chi_{\Om}$,  with the function 
$\varphi$ defined in the proof of the previous lemma. 
By  \eqref{between2:eq},
\begin{equation*} 
\|[\op_\a(a_{T}), \varphi]\|_1
\lesssim 
(\a\ell)^{d-1}.
\end{equation*}
Furthermore, by Corollary \ref{cyl:cor},
\begin{equation*}
\| \varphi\op_\a(a_{T}- \chi_{\Om})\|_1
\lesssim 
(\a\ell)^{d-1}.
\end{equation*}
Now \eqref{9Dp:eq} leads to \eqref{jump:eq}.
\end{proof}

\begin{lem} \label{ext:lem}
Suppose that $\Om$ is a bounded Lipschitz region. 
Suppose that $\R^d\setminus\L\subset B(\bold0, R_0)$ with some $R_0>0$, 
and let $\varphi\in\plainC\infty(\R^d)$ be a bounded 
function such that $\varphi(\bx) = 0$ for $\bx\in B(\bold0, R_0)$ 
and $\varphi(\bxi) = 1$ for $|\bx|> 2R_0$. 
Let $\a\gtrsim 1, 0<T\lesssim 1$. Then 
\begin{align}\label{ext:eq}
\|\varphi D_\a(a_{T}, \L; f)\|_1
\lesssim N_n(f) \a^{d-1}. 
\end{align}
The constant in \eqref{ext:eq} 
is independent of $\L$, but may depend on 
$R_0$ and $\varphi$.
\end{lem}

\begin{proof} 
We use Lemma \ref{9Dp:lem} with $\Pi = \R^d$, 
so that $\L$ and $\Pi$ satisfy  \eqref{localp:eq}. 
Thus we can use the bound \eqref{9Dp_away:eq} with 
$b = a_{T}$, and hence 
the left-hand side of \eqref{ext:eq} is bounded from above by 
\begin{equation*}
\|[\varphi, \op_\a(a_{T}) ]\|_1
 = \|[I-\varphi, \op_\a(a_{T}) ]\|_1.
\end{equation*}
The function $1-\varphi$ is compactly supported. 
Now using \eqref{between2:eq} we arrive at \eqref{ext:eq}.
\end{proof}

\subsection{Proof of Theorem \ref{main:thm} for $\a T \lesssim 1$} 
We assume that $\L$ and $\Om$ satisfy conditions of Theorem \ref{main:thm}, 
and that $\Om$ is bounded (see Remark \ref{boundom:rem}).
For brevity, in the proof 
we often use the short-hand notation 
$D_\a(f) = D_\a(a_{T}, \L; f)$ and $\GV_1 = \GV_1(1; \p\L, \p\Om)$.   
As in the previous section, we use the notation 
$g_p(t) = t^p$, $p =1, 2, \dots$. 
 
\begin{rem}\label{compact:rem}
It is clear that we can use for the operator 
$D_\a(f)$ all the bounds, established earlier, with arbitrary smooth 
functions $f$ without the assumption that $f$ is compactly supported. 
Indeed, since $\|a_T\|\lesssim 1$ 
we have $f(a_T) = f(a_T) \z(a_T)$ 
and $f(W_\a) = f(W_\a) \z(W_\a)$,\ $W_\a = W_\a(a_T; \L)$,
with some fixed 
$\z\in\plainC\infty_0(\R)$. 
In particular, this observation applies to the 
polynomial functions $g_p$. 
\end{rem}

We precede the proof with some bounds for the integral 
\eqref{U:eq}, see \cite[Lemma 4.6]{Sob_14}:  

\begin{prop}\label{co:prop}
If $f\in \plainW{1}{\infty}(\R)$, then
\begin{equation}\label{diff:eq}
|U(f)|\le 2  \|f'\|_{\plainL\infty}.
\end{equation}
If $f$ satisfies \eqref{fnorm:eq} with $n=1$ and some $0 < R\le 1$, 
then 
\begin{equation}\label{sing:eq}
|U(f)|\lesssim  
R^{\frac{\g}{2}} \1 f\1_1.
\end{equation}
\end{prop}

Now we can proceed with the proof of Theorem \ref{main:thm}. 
It follows the idea of \cite{Sob_14}, and consists of three parts: 
first we consider polynomial functions $f$, then extend it to arbitrary 
$\plainC2$-functions, and finally complete the proof for functions satisfying the 
conditions of the Theorem.

\underline{Step 1. Polynomial $f$.}  
Let $f = g_p$. 
Let $R_0$ be such that either $\L\subset B(\bold0, R_0)$ or 
$\R^d\setminus \L\subset B(\bold0, R_0)$.  
Let $\varphi\in\plainC\infty_0(\R^d)$ 
be a function such that 
$\varphi(\bx) = 1$ for $|\bx|\le R_0$, and $\varphi(\bx) = 0$ 
for $|\bx| > 2R_0$.
Since $\a R_0\gtrsim 1$ and $\a R_0 T\lesssim 1$, from Lemma \ref{jump:lem} 
we obtain that 
\begin{align*}
\|\varphi D_\a(a_{T}, \L; g_p)
- \varphi D_\a(\chi_{\Om}, \L; g_p)\|_1
\lesssim 
(\a R_0)^{d-1},
\end{align*}
with an implicit constant depending on $p$. 
In combination with Proposition \ref{IETO:prop} this gives the equality
\begin{equation}\label{varphi_as:eq}
\underset{ \substack{\a\to\infty\\ \a T\lesssim 1 }}\lim \ \frac{1}{\a^{d-1}\log\a}
\tr\bigl(\varphi D_\a(a_T, \L; g_p)\bigr)
= U(g_p) \GV_1.
\end{equation}
If $\L\subset B(\bold0, R_0)$, then $\varphi\chi_\L = \chi_\L$, and hence 
these asymptotics 
coincide with the sought 
formula \eqref{main_le:eq}. 
If $\R^d\setminus \L\subset B(\bold0, R_0)$, then 
we invoke Lemma \ref{ext:lem} which implies that 
\begin{equation*}
\underset{ \substack{\a\to\infty\\ \a T\lesssim 1 }}\lim \ \frac{1}{\a^{d-1}\log\a}
\tr\bigl((I-\varphi) D_\a(a_{T}, \L; g_p)\bigr)
= 0.
\end{equation*} 
Together with \eqref{varphi_as:eq} this gives 
\eqref{main_le:eq} for $f = g_p$ again. 

\underline{Step 2. Arbitrary functions $f\in\plainC2(\R)$.} 
The extension from polynomials to more general 
functions is done in the same way as in 
\cite{Sob_14}, and we remind this argument for the sake of 
completeness.  

Let $\z\in\plainC\infty_0(\R)$ be the function introduced before 
Proposition \ref{co:prop}. 
Let $g$ be a polynomial such that 
\[
\|(f-g)\z\|_{\plainC2} < \d.
\] 
For $g$  we can use the formula \eqref{main_le:eq} established previously:
\begin{equation}\label{fp:eq} 
\underset{ \substack{\a\to\infty\\ \a T\lesssim 1 }}\lim
\frac{1}{\a^{d-1}\log \a}
\tr D_\a(g) = U(g)\GV_1.
\end{equation}
On the other hand,  by virtue of \eqref{smooth2:eq}, we have 
\begin{align*}
\|D_\a(f-g)\|_1 = \|D_\a\big((f-g)\z\big)\|_1
\lesssim \|(f-g)\z\|_{\plainC2}\ \a^{d-1}\log\a\lesssim \d \a^{d-1}\log\a,
\end{align*}
for $\a T\lesssim 1$,  
and also, by \eqref{diff:eq}, 
\[
|U(f) - U(g)|
= 
|U(f - g)| = |U\big((f - g)\z\big)|
\le 2 \| \big((f - g)\z\big)'\|_{\plainL\infty} < 2\d.
\]
Thus, using \eqref{fp:eq} and the additivity 
\[
D_\a(f) = D_\a(g) + D_\a(f - g),\ 
U(f) = U(g) + U(f - g),
\] we get
\begin{equation*}
\underset{ \substack{\a\to\infty\\ \a T\lesssim 1 }}\limsup
\biggl|
\frac{1}{\a^{d-1}\log \a}
\tr D_\a(f)  
-  U(f)\GV_1\biggr|\lesssim \d.
\end{equation*}
Since $\d>0$ is arbitrary, we obtain 
\eqref{main_le:eq} for arbitrary $f\in\plainC2(\R)$.

\underline{Step 3. Completion of the proof.} 
Let $f$ be a function as specified in Theorem \ref{main:thm}. 
Without loss of generality 
suppose that the set $X$ consists 
of one point, and this point is $z = 0$.

Let $\z\in\plainC\infty_0(\R)$ be a real-valued function, such that 
$\z(t) = 1$ for $|t|\le 1/2$. 
Represent $f= f_R^{(1)}+ f_R^{(2)}, 0<R\le 1$, where 
$f_R^{(1)}(t) = f(t) \z\bigl(tR^{-1}\bigr)$,\ 
$f_R^{(2)}(t) = f(t)  - f_R^{(1)}(t)$. 
It is clear that  $f_R^{(2)}\in \plainC2(\R)$, 
so one can use the formula \eqref{main_le:eq} established in Step 2 of the proof:
\begin{equation}\label{g2R_le:eq}
\underset{ \substack{\a\to\infty\\ \a T\lesssim 1 }}\lim\frac{1}{\a^{d-1}\log\a}
D_\a(f_R^{(2)}) =  U(f_R^{(2)}) \GV_1.
\end{equation}
For $f_R^{(1)}$ we use Theorem 
\ref{entropy:thm} taking into account that 
$\1 f_R^{(1)}\1_2\lesssim \1 f\1_2$:
\begin{equation*}
 |\tr D_\a(f_R^{(1)})|
\lesssim 
R^{\g-\s} \1 f\1_2 \a^{d-1}\log\a,  \ \ \a\gtrsim1, \a T\lesssim 1, 
\end{equation*}
for any $\s \in (d\b^{-1}, \g)$, $\s\in (0, 1]$.
Moreover, 
by \eqref{sing:eq}, 
\begin{equation*}
|U(f_R^{(1)})|\lesssim R^{\frac{\g}{2}} \1 f\1_1.
\end{equation*}
Thus, using \eqref{g2R_le:eq} and the additivity 
\[
D_\a(f) = D_\a(f_R^{(2)}) + D_\a(f_R^{(1)}),\  
U(f) = U(f_R^{(2)}) + U(f_R^{(1)}),
\] 
we get the bound
\begin{equation*}
\limsup_{\a\to\infty }\biggl|\frac{1}{\a^{d-1}\log\a}
D_\a(f) - U(f)\GV_1\biggr| 
\lesssim \1 f\1_2 \bigl(R^{\g-\s} + R^{\frac{\g}{2}}\bigr).
\end{equation*}
Since $R$ is arbitrary, by taking $R\to 0$, we obtain 
\eqref{main_le:eq} for the function $f$.
\qed

\section{Proof of Theorem \ref{main:thm} for $\a T\gtrsim 1$}
\label{proof2:sect}

\subsection{Proof of Theorem \ref{main:thm}: basic smooth domains $\L$}
We begin with an asymptotic formula for the trace 
\begin{equation*}
\tr\bigl(
\varphi D_\a(a_{T}, \L; g_p)
\bigr),
\end{equation*}
with a basic $\plainC1$-domain $\L$, a function 
$\varphi\in\plainC\infty_0(\R^d)$, and a polynomial $g_p(t) = t^p$, 
$p = 1, 2, \dots$. 
As before we assume that $\Om$ is bounded. We assume that 
 $f\in\plainC{n}_0(\R)$ with some $n\ge 3$.
Our immediate objective is to prove the following result. 
 
\begin{thm}\label{boundary_contr:thm}
Let $\L$ be a basic $\plainC1$-domain, 
and let $\Om$ be a bounded $\plainC3$-region. 
Suppose that $\a T\gtrsim 1$, and that 
$\varphi\in\plainC\infty_0(\R^d)$. Then 
\begin{align}\label{boundary_contr:eq}
\underset{ \substack{T\to 0\\ \a T\gtrsim 1 }}\lim\ \frac{1}{\a^{d-1}\log\frac{1}{T}}
\tr \bigl(\varphi D_\a( a_{T}, \L;  g_p)\bigr) 
=  U(g_p)\GV_1(\varphi; \p\L, \p \Om).
\end{align} 
\end{thm}

By rescaling and translating we may assume that 
in Theorem \ref{boundary_contr:thm} 
the support of $\varphi$ is contained in the ball $B(\bold0, 1)$. 
We also assume that 
\begin{equation}\label{varphi1:eq}
\SN^{(d+2)}(\varphi; 1)\le 1.
\end{equation}
If the support of $\varphi$ has an empty intersection with the boundary $\p\L$, 
then by \eqref{9D_away:eq}, 
\begin{equation}\label{bound_away:eq}
|\tr \bigl(\varphi D_\a( a_{T}, \L;  f)\bigr)|
\lesssim N_n(f)\a^{d-1},
\end{equation}
and hence \eqref{boundary_contr:eq} automatically holds. 

It remains to consider the case where 
$\supp\varphi\cap\p\L\not=\varnothing$. 
For this case we construct
a convenient partition of unity. For $\L = \G(\Phi)$  let 
$\bx_{\hat\bn} = \bigl(\ell\hat\bn, \Phi(\ell\hat\bn)\bigr)$,\ 
$\hat\bn\in\mathbb Z^{d-1}$, $\ell = ~(\a T)^{-1}$, 
be the points on the boundary $\p\L$. 
Then the balls 
$B(\bx_{\hat\bn}, r)$, $\hat\bn\in\Z^{d-1}$, with 
$r = \ell\sqrt{(1+M)^2+1}$ form 
a covering of the strip 
\begin{equation*}
\L_\ell = \{\bx\in\R^d: \Phi(\hat\bx)< x_d< \Phi(\hat\bx) + \ell\}
\subset\L.
\end{equation*}
Let $\phi_{\hat\bn}\in\plainC\infty_0(\R^d)$ be a partition of unity 
subordinate to this covering. 
We may assume that 
\begin{equation}\label{hatn:eq}
\SN^{(n)}(\phi_{\hat\bn}; \ell) \lesssim 1,\ n = 0, 1, 2, \dots,
\end{equation}
uniformly in $\hat\bn\in\mathbb Z^{d-1}$. 
Denote 
\begin{equation}\label{w1w2:eq}
w_1(\bx) = \varphi(\bx)\sum_{{\hat\bn}\in\mathbb Z^{d-1}} \phi_{\hat\bn}(\bx),\ 
w_2(\bx) = \varphi - w_1(\bx).
\end{equation}
Note that
\begin{equation}\label{number:eq}
\#\{\hat\bn\in\Z^{d-1}: \bx_{\hat\bn}\in \supp\varphi\}\lesssim \ell^{1-d}.
\end{equation}

\begin{lem}\label{boundary:lem} 
Let $\L$ be a 
basic Lipschitz domain, and  let $\Om$ be a bounded Lipschitz region. 
Let $0<T\lesssim 1$, $\a T\gtrsim 1$. Then 
\begin{equation}\label{boundary1:eq}
\bigl|\tr \bigl[w_1 \bigl(D_\a(a_T, \L; f) 
- D_\a(\chi_{\Om}, \L; f)\bigr)\bigr]\bigr|
\lesssim N_{n}(f) \a^{d-1},
\end{equation}
uniformly in $\L$. 
\end{lem}

\begin{proof} 
The trace on the left-hand side 
of \eqref{boundary1:eq} coincides with 
\begin{equation*}
\sum_{\hat\bn} \tr\bigl[
\varphi\phi_{\hat\bn}
\bigl(D_\a(a_T, \L; f) - D_\a(\chi_{\Om}, \L; f)\bigr)
\bigr].
\end{equation*}
The support of $\phi_{\hat\bn}$ is contained in $B(\bx_{\hat\bn}, r)$, 
$r\asymp\ell$, and also $\a\ell = T^{-1} \gtrsim 1$, $\a\ell T = 1$.
Thus from Lemma \ref{jump:lem} for each summand we obtain 
the bound by 
\begin{equation*}
 N_{n}(f) \ \SN^{(d+2)}(\varphi\phi_{\hat\bn}; \ell)\  (\a \ell)^{d-1}. 
\end{equation*}
In view of \eqref{varphi1:eq}, 
\eqref{hatn:eq} and \eqref{number:eq},  
this leads to \eqref{boundary1:eq}. 
\end{proof}

\begin{lem}\label{away:lem} 
Let $\L$ be a basic Lipschitz domain, 
and let $\Om$ be a bounded Lipschitz region.   
Suppose that $0<T\lesssim 1$, $\a T\gtrsim 1$. Then  
\begin{equation}\label{away:eq}
\bigl|\tr \bigl(w_2 D_\a(a_T, \L; f)\bigr)\bigr|
\lesssim N_{n}(f) \a^{d-1},
\end{equation}
uniformly in $\L$.
\end{lem}

\begin{proof}
Since the function $w_2$ satisfies \eqref{localp:eq} 
with $\Pi = \R^d$, the bound \eqref{9Dp_away:eq} implies that 
the left-hand side is bounded (up to the factor $N_{n}(f)$) by 
\begin{equation}\label{prom_comm:eq}
\|[w_2, \op_\a(a_T)]\|_1 
\le \|[\varphi, \op_\a(a_T)]\|_1 
+ \sum_{\hat\bn}\|[\varphi\phi_{\hat\bn}, \op_\a(a_T)]\|_1.
\end{equation}
By \eqref{between2:eq}, 
\begin{align*}
\|[\varphi, \op_\a(a_T)]\|_1
\lesssim &\ \SN^{(d+2)}(\varphi; 1) \a^{d-1},\\[0.2cm]
\|[\varphi\phi_{\hat\bn}, \op_\a(a_T)]\|_1
\lesssim &\ \SN^{(d+2)}(\varphi\phi_{\hat\bn}; \ell) (\a \ell)^{d-1}.
\end{align*}
As in the proof of Lemma \ref{boundary:lem}, the number of summands 
does not exceed $\ell^{1-d}$, so that 
the right-hand side of \eqref{prom_comm:eq} 
is bounded by $N_n(f) \a^{d-1}$, as 
claimed.
\end{proof}
  
\begin{proof}[Proof of Theorem \ref{boundary_contr:thm}]
By Lemmas \ref{boundary:lem}, \ref{away:lem} and Remark \ref{compact:rem}, 
\begin{align}\label{varphiw:eq}
\bigl|\tr \bigl(\varphi D_\a( a_{T}, \L;  g_p)\bigr)  
- \tr \bigl(w_1 D_\a( &\ \chi_{\Om}, \L;  g_p)\bigr)\bigr|\notag\\[0.2cm]
= &\ \bigl|\tr \bigl(w_2 D_\a( a_T, \L;  g_p)\bigr)\bigr|  
\lesssim 
\a^{d-1},
\end{align}
with an implicit constant 
depending on $p$. 
Thus it suffices to find the asymptotics of the required form 
for the second term on the left-hand side. 
To analyse the asymptotics for each term in the definition 
of $w_1$( see \eqref{w1w2:eq}) we use the formula \eqref{onebasic:eq} 
with the function $\phi_{\hat\bn}\varphi$ and $R = \ell$. 
Since $\ell = (\a T)^{-1}$,\ and $\a\ell = T^{-1}$ 
the formula \eqref{onebasic:eq} rewrites as follows:
\begin{align}\label{critical1:eq}
\underset{\substack{T\to 0,\\ \a T\gtrsim 1}}
\lim \ell^{1-d} 
\biggl[
\frac{1}{\a^{d-1} \log\frac{1}{T}}
\tr \bigl(\phi_{\hat\bn}\varphi &\ 
D_\a( \chi_{\Om}, \L;  g_p)\bigr)\notag\\[0.2cm]
&\ - U(g_p) \GV_1(\phi_{\hat\bn}\varphi; \p\L, \p\Om)
\biggr] = 0,
\end{align} 
uniformly in $\hat\bn\in\mathbb Z^{d-1}$. 
Therefore
\begin{align*}
\limsup  &\ 
\biggl|
\frac{1}{\a^{d-1} \log\frac{1}{T}}
\tr \bigl( w_1D_\a( \chi_{\Om}, \L;  g_p)\bigr)
-  U(g_p) \GV_1(w_1; \p\L, \p\Om)
\biggr|\\[0.2cm]
\le &\ 
\limsup   
\underset{\hat\bn}\sum
\biggl|
\frac{1}{\a^{d-1} \log\frac{1}{T}}
\tr \bigl(\phi_{\hat\bn}\varphi D_\a( \chi_{\Om}, \L;  g_p)\bigr)
- U(g_p) \GV_1(\phi_{\hat\bn}\varphi; \p\L, \p\Om)
\biggr|\\
\lesssim &\ 
\limsup  
\ \ell^{1-d} 
\underset{\hat\bn}\max
\biggl|
\frac{1}{\a^{d-1} \log\frac{1}{T}}
\tr \bigl(\phi_{\hat\bn}\varphi 
D_\a( \chi_{\Om}, \L;  g_p)\bigr)
- U(g_p) \GV_1(\phi_{\hat\bn}\varphi; \p\L, \p\Om)
\biggr|,
\end{align*}
as $T\to 0, \a T\gtrsim 1$. 
Here we have used \eqref{number:eq} again. Given the uniformity 
in $\hat\bn\in\mathbb Z^{d-1}$ and the bound 
\eqref{number:eq}, the formula \eqref{critical1:eq} implies that 
\begin{equation*}
\underset{T\to 0}\lim  
\biggl[
\frac{1}{\a^{d-1} \log\frac{1}{T}}
\tr \bigl( w_1D_\a( \chi_{\Om}, \L;  g_p)\bigr)
- U(g_p) \GV_1(w_1; \p\L, \p\Om)
\biggr] = 0. 
\end{equation*} 
Taking into account \eqref{varphiw:eq} and that 
$w_1(\bx) = \varphi(\bx)$ for $\bx\in\p\L$, this leads to  
\eqref{boundary_contr:eq}.   
\end{proof}

\subsection{Proof of Theorem \ref{main:thm}: basic piece-wise 
smooth domains $\L$}

Here we extend the formula \eqref{boundary_contr:eq} 
to piece-wise smooth domains $\L$. Our argument follows the 
proof of \cite[Theorem 4.1]{Sob2}. 
For simplicity we assume that only $\L$ is piece-wise smooth, whereas 
$\Om$ remains smooth. This simplification preserves the idea 
of \cite{Sob2}, but allows one to avoid some routine technical work 
that would have been just a modified repetition of the proof 
from \cite{Sob2}. 
 
\begin{thm}\label{pw:thm}
 Let $\L$ be a basic piece-wise $\plainC1$-domain, and let $\Om$ 
 be a bounded $\plainC3$-region. Suppose that  
 $\varphi\in\plainC\infty_0(\R^d)$. Then the formula \eqref{boundary_contr:eq} 
 holds. 
 \end{thm} 

\begin{proof}
Assume as before that $\varphi(\bx) = 0$ for $|\bx| >1$ and 
\eqref{varphi1:eq} holds. 
We follow the idea of the proof of \cite[Theorem 4.1]{Sob2}.
Cover the ball $B(\bold0, 1)$ with open 
balls of radius $\varepsilon>0$, such that the number of intersections 
is bounded from above uniformly in $\varepsilon$. 
Introduce a  subordinate partition of unity $\{\phi_j\}, j = 1, 2, \dots$, 
such that $\SN^{(n)}(\phi_j; \varepsilon)\lesssim 1$ uniformly in $j= 1, 2, \dots$.   
By \eqref{bound_away:eq} contributions 
to \eqref{boundary_contr:eq} from the balls having empty intersection 
with $\p\L$, equal zero. 

Let $\Sigma$ be the set of indices such 
that the ball indexed by $j\in\Sigma$ has a non-empty intersection 
with the set $(\p\L)_{\rm s}$, see Subsect. \ref{domains:subsect} for the definition. 
Since the set $(\p\L)_{\rm s}$ is built out of $(d-2)$-dimensional Lipschitz surfaces, 
we have 
\begin{equation}\label{sigma:eq}
\#\Sigma\lesssim \varepsilon^{2-d}.
\end{equation}
By \eqref{log_bound:eq},  
for each ball we have the bound 
\begin{equation*}
\bigl|\tr\bigl(\varphi\phi_jD_\a(a_{T}, \L; g_p)\bigr)\bigr|
\lesssim 
(\a \varepsilon)^{d-1} \log\frac{1}{T},
\end{equation*}
with an implicit constant depending on $p$,
uniformly in $j = 1, 2, \dots$, if $\a\varepsilon\gtrsim 1$. 
By virtue of \eqref{sigma:eq}, this implies that
\begin{equation*}
\underset{j\in\Sigma}\sum 
\biggl|\tr\bigl(\varphi\phi_jD_\a(a_{T}, \L; g_p)\bigr)\biggr|
\lesssim 
\varepsilon \a^{d-1}\log\frac{1}{T},\ \textup{if}\ \ \a\varepsilon\gtrsim 1.
\end{equation*}
As 
\begin{equation*}
\underset{j\in\Sigma}\sum 
\bigl|
\GV_1(\varphi\phi_j, \p\L; \p\Om)
\bigr|\lesssim \varepsilon
\end{equation*}
(see \eqref{W1:eq} for the definition of $\GV_1$), 
we can rewrite the last two formulas as follows:
\begin{align}\label{vare:eq}
\underset{\substack{T\to 0\\
\a T\gtrsim 1}}\limsup\underset{j\in\Sigma}\sum 
\ \biggl|\frac{1}{\a^{d-1}\log\frac{1}{T}} 
\tr\bigl(\varphi\phi_jD_\a(a_{T}, \L; g_p)\bigr)
 - U(g_p) \GV_1(\varphi\phi_j, \p\L, \p\Om)\biggr|
 \lesssim \varepsilon.
\end{align}
Let us now turn to the balls 
with indices $j\notin\Sigma$. We may assume that they are separated from 
$(\p\L)_{\rm s}$. Thus in each such ball the boundary of $\L$ is $\plainC1$. 
By \eqref{9D:eq}, we may assume that the entire 
$\L$ is $\plainC1$, and hence  
Theorem \ref{boundary_contr:thm} is applicable. Together with \eqref{vare:eq}, 
this gives 
\begin{align*}
 \underset{\substack{T\to 0\\
\a T\gtrsim 1}}\limsup\biggl|\frac{1}{\a^{d-1}\log\frac{1}{T}} 
\tr\bigl(\varphi D_\a(a_{T}, \L; g_p)
 - U(g_p) \GV_1(\varphi, \p\L, \p\Om)\biggr|
 \lesssim \varepsilon.
\end{align*}
Since $\varepsilon>0$ is arbitrary, this proves the Theorem.
\end{proof}

\subsection{Proof of Theorem \ref{main:thm}: completion} 

Theorem \ref{pw:thm} 
extends to arbitrary piece-wise $\plainC1$ region $\L$ by 
using the standard partition of unity 
argument based on Lemma \ref{9D:lem}. 
Also, as mentioned earlier, 
one can extend 
Theorem \ref{pw:thm} to the piece-wise
 $\plainC3$-regions $\Om$.  We omit this argument since it 
 repeats the proofs in \cite{Sob2}.

The extension of Theorem \ref{pw:thm} to arbitrary 
functions $f$ specified in Theorem \ref{main:thm}, 
is done in the same way as in 
the proof of Theorem \ref{main:thm} for $\a T \lesssim 1$, with the help 
of the  
bounds \eqref{entropy:eq} and \eqref{smooth2:eq}. We omit the details. 
\qed

\section{Comparison with known asymptotic formulas} 
\label{compar:sect}

\subsection{Coefficient $\CB_d$}
The asymptotics of $\tr D_\a(a, \L; f)$ 
for a fixed symbol $a$ as $\a\to\infty$ have 
been studied rather extensively in 
the 1980's by H. Widom and R. Roccaforte, e.g. \cite{Widom_85} 
and references therein. It is interesting and 
instructive to compare there findings with 
the asymptotics in Theorem \ref{main:thm}. 
As shown in \cite{Widom_85}, \cite{BuBu}, 
for a fixed smooth symbol $a$, smooth $f$ and 
smooth $\L$ one can write out a complete asymptotics 
expansion in powers of $\a^{-1}$. 
In this section we focus only on the first 
coefficient of this expansion that we denoted by 
$\CB_d(a; \p\L, f)$ in the Introduction, see \eqref{bd:eq}.
 It has a more complicated form than 
the coefficient $\GV_1$ defined in \eqref{W1:eq} and featuring in 
Theorem \ref{main:thm}, and it is described below. 
 
For a function $f: \mathbb C\to \mathbb C$ 
and any $s_1, s_2\in\mathbb C$ define the integral  
\begin{equation}\label{Uf:eq}
U(s_1, s_2; f) 
= \int_0^1 \frac{f\bigl((1-t)s_1 + t s_2\bigr) 
	- [(1-t)f(s_1) + t f(s_2)]}{t(1-t)} dt.
\end{equation}
The integral $U(f)$ defined in \eqref{U:eq} is easily 
expressed as $U(f) = U(1, 0; f)$. 
It is clear that $U(s_1, s_2; 1) = U(s_1, s_2; t) = 0$,\ 
for all $s_1, s_2\in\mathbb C$.
This integral is  
finite for functions $f\in \plainC{0, \varkappa}(\mathbb C)$,  
$\varkappa\in (0, 1]$. 
It is also H\"older-continuous:
for any $\d\in (0, \varkappa)$ we have 
\begin{equation}\label{approx1:eq}
|U(s_1, s_2; f) - U(r_1, r_2; f)|\lesssim  \textup{Lip}_\varkappa
(f)\bigl(|s_1-r_1|^\d + |s_2-r_2|^\d\bigr), 
\end{equation}
where
\begin{equation*}
\textup{Lip}_\varkappa(f) 
= \sup_{z\not = w} \frac{|f(z) - f(w)|}{|z-w|^\varkappa}.
\end{equation*}
Note also that 
\begin{equation}\label{U_sym:eq}
U(s_1, s_1; f) = 0,\ \ \ 
U(s_1, s_2; f) = U(s_2, s_1; f), \forall s_1, s_2\in \mathbb C.
\end{equation}
For a symbol $a = a(\xi),\ \xi\in \R$ define 
\begin{equation}\label{cb:eq}
	\CB(a; f) := \frac{1}{8\pi^2}\lim_{\varepsilon\to 0}
	\underset{|\xi_1-\xi_2|>\varepsilon}\iint  
	\frac{U\bigl(a(\xi_1), a(\xi_2); f\bigr)}{|\xi_1-\xi_2|^2}
	d\xi_1 d\xi_2.
\end{equation} 
If $f$ is smooth, then this definition coincides with the standard 
double integral. In particular, if $f''$ is bounded, then 
\begin{equation*}
|\CB(a; f)|\lesssim \|f''\|_{\plainL\infty} \iint 
\frac{|a(\xi_1) - a(\xi_2)|^2}{|\xi_1 - \xi_2|^2}
d\xi_1 d\xi_2.
\end{equation*}
This estimate was first pointed out in \cite{Widom_82}. 
Note that  $\CB$ 
is invariant under the change $a(\xi) \to a(\tau \xi)$ 
with an arbitrary $\tau >0$. If $f$ is allowed 
to be non-smooth, as in Condition \ref{f:cond}, then 
the finiteness of the limit in \eqref{cb:eq} is not trivial, and 
we comment on this later, in Proposition \ref{scales:prop}. 

As shown in \cite{Widom_82}, see also \cite{LeSpSo_15}, 
in the case $d=1$, for smooth $f$ and $a$ 
we have $\tr D_\a(a, \R_+; f) \to \CB(a; f)$ 
as $\a\to\infty$. 
For the multi-dimensional case the asymptotic 
coefficient $\CB_d(a;\p\L, f)$ is defined as follows. 
For a unit vector $\be\in\R^d, d\ge 2$, introduce the hyperplane 
\begin{equation*}
\Pi_{\be} := \{ \bxi\in\R^d: \be\cdot\bxi = 0 \}.
\end{equation*} 
Introduce the orthogonal coordinates 
$\bxi = (\overc{\bxi}, t)$ such that $\overc{\bxi}\in \Pi_\be$ and $t\in\R$. 
Then we set
\begin{equation}\label{cbd:eq}
	\CB_d(a; \p\L,  f) := \frac{1}{(2\pi)^{d-1}}
	\int_{\p\L}   \CA_d(a, \bn_\bx; f) dS_\bx,\ \ 
	\CA_d(a, \be; f) := 
	\int_{\Pi_{\be}} 
	\CB\bigl(a(\overc{\bxi}, \ \cdot\ ); f\bigr)d\overc{\bxi}.
	\end{equation} 
	For the smooth symbol $a$ and smooth function $f$ 
it was proved by H. Widom 
(see \cite{Widom_80} and \cite{Widom_85}), that 
the trace of $D_\a(a, \L; f)$ satisfies 
\eqref{bd:eq}. 
%
Clearly, the formula \eqref{bd:eq} 
describes the asymptotics of $\tr D_\a$ 
for the symbol $a = a_{T}$, 
as $\a\to\infty$ and $T>0$ is fixed. 
On the other hand, Theorem \ref{main:thm}
offers an asymptotic formula in two parameters: $\a\to\infty$ 
and $T\to 0$. Our aim now is to compare the asymptotic coefficient 
defined in \eqref{W1:eq}, with the 
coefficient $\CB_d(a_T, \p\L; f)$ as $T\to 0$. 
The relevant calculations are quite involved, 
and to avoid further complications, 
we assume that $\Om$ is smooth. 
In fact, for these purposes it will be 
sufficient to assume that $\Om$ is $\plainC2$-smooth.

\begin{thm} \label{comparison:thm}
Let $\L$ satisfy Condition \ref{domain:cond}, 
and let $\Om\subset \R^d$ be a bounded $\plainC2$-region. 

Let the function $f$ satisfy Condition \ref{f:cond} 
with $n = 2$ and some $\g >0$. 

Let $a = a_{T}$ be the symbol defined in 
Subsection \ref{at:subsect} with some $\b > \max\{d\g^{-1}, d\}$. 

Then 
\begin{align}\label{comparison:eq}
\underset{T\to 0}
\lim\ \frac{1}{\log\frac{1}{T}} \CB_d(a_{T}; \p\L, f) 
= U(f)\GV_1(1; \p\L, \p\Om).
\end{align}
\end{thm}

As pointed out in the Introduction, 
due to this theorem, the formula \eqref{main_gr:eq} 
can be rewritten in the form \eqref{standard:eq}, 
and hence it can be viewed as an extension of \eqref{bd:eq} 
to the asymptotics in two parameters, $\a$ and $T$. 
Such an asymptotic formula 
was obtained in \cite{LeSpSo_15} for 
the one-dimensional case. 
Note that \eqref{main_le:eq} cannot be rewritten in the same way. 

\subsection{Coefficient $\CB_d$ for the symbol $a_T$} 
We begin our analysis of the coefficient $\CB_d(a;\p\L, f)$ with 
studying the multi-scale symbols introduced in 
Subsection \ref{mscale:subsect}. Let the symbol 
$a$ satisfy \eqref{scales:eq} 
with the scale $\tau$ and amplitude $v$ that satisfy \eqref{Lip:eq} 
and \eqref{w:eq} respectively. Assume that \eqref{tauinf:eq} 
holds. As the region $\L$ is always fixed, 
for brevity we omit $\p\L$ from the notation and write simply 
$\CB_d(a; f)$. 

From now on we assume that $f$ satisfies Condition 
\ref{f:cond} with some $\g >0$ and $n=2$. 
We use the notation $\varkappa = \min\{1, \g\}$. 
The next proposition is borrowed from 
\cite[Theorem 6.1]{Sob_15}.

\begin{prop}\label{scales:prop}
Suppose that $f$ satisfies Condition \ref{f:cond} with $n = 2$, 
$\g >0$ and some $R > 0$. 
Let the symbol 
$a\in\plainC\infty(\R)$ be a real-valued symbol described above. 
Then for any $\s\in (0, \varkappa]$ we have 
\begin{equation}\label{coeffscales:eq}
|\CB(a; f)|\lesssim \1 f\1_2 R^{\g-\s} V_{\s, 1}(v, \tau), 
\end{equation}
with a constant independent of $f$, 
uniformly in the functions $\tau, v$, 
and the symbol $a$.
\end{prop}

We note another useful result from \cite{Sob_15}. 
It describes the contribution of ``close" points 
$\xi_1$ and $\xi_2$ in the coefficient \eqref{cb:eq}. For $r>0$  define 
\begin{align}
\CB^{(1)}(a; f; r) = &\ \frac{1}{8\pi^2}\lim_{\varepsilon\to 0}
	\underset{\varepsilon<|\xi_1-\xi_2| < r}\iint  
	\frac{U\bigl(a(\xi_1), a(\xi_2); f\bigr)}{|\xi_1-\xi_2|^2}
	d\xi_1 d\xi_2,\label{cb1:eq}\\[0.2cm]
\CB^{(2)}(a; f; r) = & \frac{1}{8\pi^2}
\underset{|\xi_1-\xi_2| \ge r}\iint  
	\frac{U\bigl(a(\xi_1), a(\xi_2); f\bigr)}{|\xi_1-\xi_2|^2}
	d\xi_1 d\xi_2.\label{cb2:eq} 
\end{align}
The integral \eqref{cb1:eq} 
is estimated in the following proposition.

\begin{prop}\label{sub:prop} 
Suppose that $f$ satisfies Condition \ref{f:cond} 
with $n = 2$, $\g>0$ and some $R>0$. 
Let $a\in\plainC\infty(\R)$ be as above. 
Suppose also that $r\le \tau_{\textup{\tiny inf}}/2$.
Then for any $\d\in [0, \varkappa)$,
the following bound holds:
\begin{equation}\label{sub:eq}
|\CB^{(1)}(a; f; r)|\lesssim \1 f\1_2 R^{\g-\varkappa}
r^\d V_{\varkappa, 1+ \d}(v, \tau),
\end{equation}
uniformly in the functions $\tau, v$, 
and the symbol $a$.
 \end{prop} 
This bound follows from \cite[Corollary 6.5]{Sob_15}.  
 
For the case $d\ge 2$, 
using the notations \eqref{cb1:eq} and \eqref{cb2:eq} define 
\begin{equation}\label{ak:eq}
\CA_d^{(k)}(a, \be; f; r) := 
	\int_{\Pi_{\be}} 
	\CB^{(k)}\bigl(a(\overc{\bxi}, \ \cdot\ ); f; r\bigr)
	d\overc{\bxi},\ \ 
	\quad k = 1, 2.
\end{equation}
We can estimate the quantities $\CA_d$ and $\CA_d^{(1)}$ similarly to 
\eqref{coeffscales:eq} and \eqref{sub:eq}:
\begin{equation}\label{cad:eq}
|\CA_d(a, \be; f)|\lesssim \1 f\1_2  V_{\vark, 1}(v, \tau),
\end{equation}
and 
\begin{equation}\label{cad1:eq}
|\CA_d^{(1)}(a, \be; f; r)|
\lesssim  \1 f\1_2
r^\d 
V_{\varkappa, 1+\d}(v, \tau), \ \forall\d\in [0, \varkappa),\ 
 r \le\frac{\tau_{\textup{\tiny inf}}}{2},
\end{equation}
uniformly in $\tau$, $v$, $a$ and $\be\in\mathbb S^{d-1}$. Indeed, 
the symbol $b(t) := a(\overc{\bxi}, t)$ satisfies conditions 
\eqref{scales:eq} and \eqref{w:eq} with 
the amplitude $v_\be(t) = v(\overc{\bxi}, t)$ and the scaling 
function $\tau_\be(t) = \tau(\overc{\bxi}, t)$. 
By \eqref{coeffscales:eq},
\begin{equation*}
|\CB(b; f)|\lesssim \1 f\1_2 V_{\vark, 1}(v_\be, \tau_\be),
\end{equation*}
uniformly in $\be$, and in the functions $v, \tau$. 
Integrating over $\overc{\bxi}$ we get \eqref{cad:eq}. 

Furthermore, 
since $\inf\tau_\be\ge \tau_{\textup{\tiny inf}}$, from 
\eqref{sub:eq} we get that 
\begin{equation*}
|\CB^{(1)}(b; f; r)|\lesssim \1 f\1_2 r^\d 
V_{\varkappa, 1+\d}(v_\be, \tau_\be),\ 
\forall r \le \frac{\tau_{\textup{\tiny inf}}}{2}, 
\end{equation*}
for any $\d\in [0, \varkappa)$, uniformly 
in $\be, v, \tau$, as above. 
Integrating over $\overc{\bxi}$, we get \eqref{cad1:eq}. 

As we have pointed out previously, 
in view of \eqref{at:eq}, the symbol $a_T$ 
satisfies \eqref{scales:eq} with the functions 
$v$ and $\tau$ defined in \eqref{vtau:eq}. Recall also that 
$\tau_{\textup{\tiny inf}}\asymp T$. Together with \eqref{vlog:eq} 
and \eqref{vp:eq}, the estimates \eqref{cad:eq} and \eqref{cad1:eq} yield the bounds 
\begin{equation}\label{ca_bound:eq}
|\CA_d(a, \be; f)|\lesssim \1 f\1_2 \log\frac{1}{T} 
\end{equation}
and
 \begin{equation}\label{ca1_bound:eq}
 |\CA_d^{(1)}(a, \be; f; r)|\lesssim \1 f\1_2,\  
 \forall r \le \frac{\tau_{\textup{\tiny inf}}}{2},
 \end{equation}
uniformly in $\be\in\mathbb S^{d-1}$. 
From now on we take $r = \t T$ where $\t >0$ is chosen 
to satisfy $r\le  \frac{\tau_{\textup{\tiny inf}}}{2}$. 
The parameter $\t$ is fixed, and we are not concerned 
about the dependence of the forthcoming estimates on $\t$.

Let us now take care of the integral $\CA_d^{(2)}$.  
 
\begin{lem} 
Let $r=\t T$, with a $\t>0$ described above. Then 
for all $0<T\lesssim 1$ we have 
\begin{align}\label{a2_diff:eq}
\bigl|\CA_d^{(2)}(a_T, \be; f; \t T) 
- \CA_d^{(2)}(\chi_\Om, \be; f; \t T)\bigr| \lesssim \1 f\1_1,
\end{align} 
uniformly in $\be\in\mathbb S^{d-1}$.
\end{lem}
 
\begin{proof} 
By definitions \eqref{cb2:eq} and \eqref{ak:eq}, 
the bounds \eqref{hf:eq} and \eqref{approx1:eq} 
imply that 
\begin{align}\label{vsp:eq}
\bigl|\CA_d^{(2)}(a_T, \be; f; \t T) 
- \CA_d^{(2)}(\chi_\Om, \be; f; \t T)\bigr|
\lesssim &\ \1 f\1_1 
\underset{\R^{d-1}}\int\underset{|t-s|> \t T}\iint
\frac{|a_T(\overc\bxi, t) - \chi_{\Om}(\overc\bxi, t)|^\d}{|s-t|^2}
ds dt d\overc\bxi\notag\\[0.2cm]
\lesssim &\ T^{-1} \1 f\1_1\int
|a_T(\bxi) - \chi_{\Om}(\bxi)|^\d d\bxi,
\end{align} 
for any $\d\in (0, \varkappa)$. 
By \eqref{ld:eq}, for $\d \in (d\b^{-1}, \varkappa)$ 
the integral on the right-hand side is finite and it 
does not exceed $T$, whence \eqref{a2_diff:eq}.
\end{proof}

\section{Coefficient $\CA_d^{(2)}(\chi_\Om, \be; f; \t T)$}
\label{coeff:sect}  
  
 \subsection{Smooth surfaces} 
In order to study the integral 
$\CA_d^{(2)}$ we need to investigate the following model problem.   
For $\rho>0$, let 
 \begin{equation*}
 	\CC^{(n)}_{\rho} := (-2\rho, 2\rho)^n	
 \end{equation*}
 be the $n$-dimensional cube. 
 We use two ways of labelling the coordinates:
 \begin{align}
 \bxi = &\ (\hat{\bxi}, \xi_d),\ 
 \hat{\bxi} := (\xi_1, \xi_2, \dots, \xi_{d-1}),\notag\\[0.2cm]
 \bxi = &\ (\overc \bxi, \xi_l),\ \overc\bxi := (\xi_1, \dots, \xi_{l-1}, \xi_{l+1},\dots \xi_d),\label{overcbxi:eq}
 \end{align}
 where $l = 1, 2, \dots, d$. If $l = d$, then, clearly, $\hat{\bxi} = \overc \bxi$. 
 Let $\Psi\in\plainC2\bigl(\overline{\CC^{(d-1)}_{\rho}}\bigr)$ be a function 
 with values in the interval $(-2\rho, 2\rho)$. 
 We focus on the surface 
 \begin{equation}\label{surface:eq}
 	S = S(\Psi) 
 	:= \{ \bxi\in\CC^{(d)}_\rho:  \xi_d = \Psi(\hat\bxi)\}.
 \end{equation}
  For each $\overc\bxi\in \CC^{(d-1)}_{\rho}$ 
 define the set 
 \begin{equation}\label{sxl:eq}
 	\SX_l(\overc\bxi) := \SX_l(\overc\bxi; \Psi)
 	:= \{\xi_l \in (-2\rho, 2\rho): (\overc\bxi, \xi_l)\in S\}.
 \end{equation}
Let us record some useful facts about the set $\SX_l(\overc\bxi)$:

\begin{lem}\label{SX:lem}
Let $\be_l$ be the unit basis vector along the direction $l$, 
and let $\bn_{\bxi}$ be a unit  
normal to $S$ at the point $\bxi\in S$.  
For any function $u$, continuous on the cube $\overline{\CC_\rho^{(d)}}$ 
and any $l=1,2,\ldots, d$ we have 
\begin{equation}\label{SX1:eq}
 	\underset{\CC_\rho^{(d-1)}}\int
 	 \ \sum_{\bxi: \xi_l \in \SX_l(\overc\bxi)} u(\bxi) 
 	 d\overc\bxi	 	
 	= \underset{S}\int u (\bxi)
 	|\bn_{\bxi}\cdot \be_l| \,d S_{\bxi}.
 	\end{equation}
In particular, the 
function $\#\bigl(\SX_l(\overc\bxi)\bigr)$ counting the number of elements 
is finite for a.e. $\overc\bxi\in\CC_\rho^{(d-1)}$ and 
\begin{equation}\label{SX:eq}
 	\underset{\CC_\rho^{(d-1)}}\int 
 	  \#\bigl(\SX_l(\overc\bxi)\bigr)\, d\overc\bxi
 	= \underset{S}\int  
 	|\bn_{\bxi}\cdot \be_l|\, d S_{\bxi}.
 	\end{equation}
\end{lem}

\begin{proof} 
Equality \eqref{SX:eq} follows from \eqref{SX1:eq} with $u\equiv 1$. 

Let us prove now \eqref{SX1:eq}. We denote by $\Xi:\CC_\rho^{(d-1)}\to\CC_\rho^{(d-1)}$ the mapping
 	\begin{equation*}
 		\Xi(\hat\bxi) :=
 		\bigl(\xi_1, \dots, \xi_{l-1}, \xi_{l+1}, 
 		\dots, \xi_{d-1}, \Psi(\hat\bxi)\bigr).
 	\end{equation*}
 	Then by the Change of Variables Formula (see, e.g., \cite[Theorem 2, p. 99]{EG}), 
        for any continuous function $u$ on the cube 
 	$\overline{\CC_{\rho}^{(d)}}$, we have
 	\begin{equation}\label{change_var:eq}	
 	 \underset{\CC_\rho^{(d-1)}}\int
 	 \ \sum_{\bxi: \xi_l \in \SX_l(\overc\bxi)} u(\bxi) 
 	 d\overc\bxi 
 	 = \underset{\CC_\rho^{(d-1)}}\int 
 	 \sum_{\hat\bxi\in \Xi^{-1}(\overc\bxi)} u(\hat\bxi, \Psi(\hat\bxi)) 
 	 d\overc\bxi	
 	 = \underset{\CC_\rho^{(d-1)}}\int u(\hat\bxi, \Psi(\hat\bxi)) 
 	 J_{\Xi}(\hat\bxi) 
 		d\hat\bxi,
 	\end{equation}
 	where $J_{\Xi}$ is the Jacobian of the map $\Xi$. 
A direct calculation shows that 
\begin{equation*}
J_{\Xi} = |\p_{\xi_l} \Psi| 
= 
\frac{|\p_{\xi_l} \Psi|}{
\sqrt{1+|\nabla\Psi|^2}}
\sqrt{1+|\nabla\Psi|^2} = |\bn_{\bxi}\cdot \be_l|
\sqrt{1+|\nabla\Psi|^2}.
\end{equation*}
Thus the equality \eqref{change_var:eq} becomes \eqref{SX1:eq}. 
\end{proof}

 \subsection{Function $m$} 
For $\overc\bxi\in \CC_\rho^{(d-1)}$ 
define the following function $m = m(\overc\bxi)$:
if $\SX_l(\overc\bxi) = \varnothing$, then we define
 \begin{equation}\label{mempty:eq}
 	m(\overc\bxi) := (4\rho)^{-1}, 
 \end{equation}
 If $\#\bigl(\SX_l(\overc\bxi)\bigr) = N\ge 1$, 
 then we label the points $\xi\in\SX_l(\overc\bxi)$ 
 in increasing order:
 $\xi^{(1)} < \xi^{(2)} < \dots $ $ < \xi^{(N)}$, and define
 \begin{equation}\label{rhok:eq}
 	m(\overc\bxi) := \sum_{j=1}^N \rho_j(\overc\bxi)^{-1},\ \
 	\begin{cases}
 		\rho_j(\overc\bxi) := \dist\{\xi^{(j)}, 
 		\SX_l(\overc\bxi)\setminus\{\xi^{(j)}\}\},
 		\ N\ge 2, \\[0.2cm]
 		\rho_1 := 4\rho,\ N = 1.
 	\end{cases}
 \end{equation}
The function $m$ is well-defined since 
$\#\bigl(\SX_l(\overc\bxi)\bigr)<\infty$ for
a.e. $\overc\bxi\in\CC_\rho^{(d-1)}$. 

In all the bounds obtained below the constants are independent 
of the function $\Psi\subset\plainC2$ and parameter $\rho >0$.
 We begin with an estimate for  $m(\overc\bxi)$, 
 see \cite[p. 185]{Widom_90} or \cite[Chapter 13]{Sob}.

 \begin{prop}\label{bigdim:prop}
 	Let $\Psi\in\plainC2\bigl(\overline{\CC^{(d-1)}_{\rho}}\bigr)$, $d\ge 2$, 
 	with some $\rho >0$.
 	Then the function $m$ satisfies the bound
 		\begin{equation}\label{reciprocal:eq}
 			\underset{\CC^{(d-1)}_{\rho}} \int m(\overc\bxi) 
 			d\overc\bxi\lesssim \rho^{d-2}
 			\bigl(1+\rho\|\nabla^2\Psi\|_{\plainL\infty}\bigr). 
 		\end{equation}
 \end{prop}
  
 Here is a useful consequence of this proposition:
 
 \begin{lem} \label{Cheb:lem} 
 Let the function $\Psi$ be as in Proposition \ref{bigdim:prop}. 
 	Let $\rho_j(\overc{\bxi})$ be the distances 
 	defined in \eqref{rhok:eq} for the set 
 	$\SX_l(\overc{\bxi})$. 
 	Let 
 	$\rho_0(\overc{\bxi}) = \inf_j \rho_j(\overc{\bxi})$ 
 	if $\#\bigl(\SX_l(\overc\bxi)\bigr)\ge 1$, 
 	and $\rho_0(\overc{\bxi}) = 4\rho$, if $\SX_l(\overc\bxi) = \varnothing$. 
 	For all $\varepsilon\in (0, 4\rho)$ denote 
 	\begin{equation}\label{meps:eq}
 		M_\varepsilon := M_\varepsilon(\Psi) :=
 		\{\overc{\bxi}\in \CC^{(d-1)}_\rho : \rho_0(\overc{\bxi}) <\varepsilon\} 
 	\end{equation}
 	Then   
 	\begin{equation}\label{Cheb:eq}
 		| M_\varepsilon|\lesssim \varepsilon \rho^{d-2}
 		\bigl(1+\rho\|\nabla^2\Psi\|_{\plainL\infty}\bigr).
 	\end{equation}
 \end{lem}
 
 \begin{proof} 
 By the definition \eqref{mempty:eq}, \eqref{rhok:eq}, 
 	\begin{equation*}
 		m(\overc{\bxi}) \ge 
 		\rho_0(\overc{\bxi})^{-1}, 
 	\end{equation*}
 	so by virtue of Chebyshev's inequality, 
we get from \eqref{reciprocal:eq} that
\begin{equation*}
\varepsilon^{-1}
\underset{M_{\varepsilon}}
\int d\overc\bxi
\le \underset{M_{\varepsilon}}\int m(\overc\bxi) d\overc\bxi
\le C\rho^{d-2}
\bigl(1+\rho\|\nabla^2\Psi\|_{\plainL\infty}\bigr).
\end{equation*} 	
This leads to \eqref{Cheb:eq}.  
 \end{proof}
 
It immediately follows from the above lemma that 
\begin{align}\label{Mmeasure:eq}
\#\bigl(X_l(\overc\bxi)\bigr)<\infty,\ 
\forall \overc\bxi\notin M(\Psi) \equiv \bigcap_{\varepsilon>0} M_\varepsilon(\Psi), 
\ \ \ \textup{and}\ \ 
|M(\Psi)| = \lim_{\varepsilon\to 0}|M_{\varepsilon}(\Psi)| = 0. 
\end{align}

 \subsection{Asymptotics 
 of $\CA_d^{(2)}(\chi_\Om, \be; f; \t T)$}
In order to use the conclusions of 
Lemma \ref{Cheb:lem}, we adopt the following conventions. 
 
\begin{itemize}
\item
We always assume that the Cartesian coordinates 
in $\R^d$ are chosen in such a way 
that the unit vector $\be$ coincides 
with the basis vector $\be_l$, 
so that the vector $\overc\bxi$ featuring in 
\eqref{cbd:eq} is given by \eqref{overcbxi:eq}.

\item
In each set $D_j$ of the covering 
\eqref{coverrd:eq}, we re-label the remaining coordinates 
to ensure that the part of the surface $\p\Om$ inside 
$D_j$ is given by the 
equality $\xi_d = \Psi_j(\hat\bxi)$. 
The coordinates $\hat\bxi$ and $\overc\bxi$ 
do not necessarily coincide, and the choice of 
$\hat\bxi$ may be different 
in different $D_j$'s. 
\end{itemize} 
Denote by $\{\phi_j\}, \tilde\phi$ 
a partition of unity subordinate to the 
covering \eqref{coverrd:eq}. 
Thus we split $\CA_d^{(2)}(\chi_\Om, \be; f; \t T)$ 
into the sum
\begin{equation*}
\CA_d^{(2)}(\chi_\Om, \be; f; \t T)
= \sum_j \CK_j(T) + \tilde \CK(T),
\end{equation*} 
where
\begin{equation}
\CK_j(T) := \frac{1}{8\pi^2}\underset{\R^{d-1}}\int \ \ 
\underset{\t T < |t-s|}\iint\phi_j(\overc\bxi, t)
\frac{U\bigl(\chi_\Om(\overc\bxi, t), \chi_\Om(\overc\bxi, s); f\bigr)}
{|t-s|^2}  ds dt \ d\overc\bxi,
\end{equation} 
 and the integral $\tilde \CK(T)$ is defined in a similar way with the 
 function $\tilde\phi$.

The properties $U(0, 0; f) = U(1, 1; f) = 0$ and $U(1, 0; f) = U(0, 1; f)$  
 allow one to rewrite $\CK_j(T)$, $\tilde\CK(T)$ in the form  
 \begin{equation}\label{ck_improved:eq}
 \CK_j(T) = \frac{U(1, 0; f)}{4\pi^2} 
 \underset{\R^{d-1}}\int \CS_j(\overc\bxi; T)   d\overc\bxi,\ 
 \tilde\CK(T) = 
 \frac{U(1, 0; f)}{4\pi^2} 
 \underset{\R^{d-1}}\int \tilde \CS(\overc\bxi; T)   d\overc\bxi,
 \end{equation}
 where
 \begin{equation*}
 \CS_j(\overc\bxi; T) := 
  \underset{t\notin \Om({\overc\bxi})}\int\phi_j(\overc\bxi, t)\ 
  \underset{s\in\Om({\overc\bxi}): 
  \t T< |t-s|}
 \int \frac{1}{|s-t|^2} ds dt,
 \end{equation*}
 with
 \begin{equation*}
 \Om({\overc\bxi}) := \{t\in\R: (\overc\bxi, t)\in \Om\},
 \end{equation*}
 and $\tilde\CS(\overc\bxi; T)$ is defined in a similar way with the function 
 $\tilde\phi$.
 Since the set $\tilde D$ is separated from 
the surface $S$, the integral 
 $\tilde\CS(\overc\bxi; T)$ is taken over the 
 set where $|s|\lesssim 1$, $|s-t|\gtrsim 1$, 
 and therefore
\begin{equation}\label{ktilde:eq}
\tilde\CK(T)\lesssim 1, 
\end{equation}
 for all $T>0$.
Let us now analyse $\CS_j(\overc\bxi; T)$. First we quote 
the following elementary statement, proved in  \cite[Lemma 9.4]{LeSpSo_15}.   

\begin{prop} \label{log_int:prop}
Let $J_k = (s_k, t_k)\subset \R$, $k = 1, 2, \dots, N$ 
be a finite collection of open intervals, 
such that their closures are pairwise disjoint, and let 
$J = \cup_k J_k$. 
Suppose that $0<T \lesssim 1$ and 
$|J_k|\le d_1$,  
$k = 1, 2, \dots, N$, with some $d_1> 0$.  
Then 
\begin{equation}\label{log_bound_up:eq}
\sum_{k=1}^N\underset{t\notin J}
\int\ \ 
\underset{|t-s|\ge T, s\in J_k}
\int \frac{1}{|t-s|^2} ds dt\lesssim N\,\log\biggl(\frac{1}{T} + 1\biggr),
\end{equation}
with a constant depending only on $d_1$. 

Assume in addition that 
\begin{equation*}
|J_k|\ge d_0,\ k = 1, 2, \dots, N,\ \ 
\ \min_{j\not=k}\dist\{J_k, J_j\}\ge d_0.
\end{equation*}
with some $d_0\in (0, d_1]$. 
Let $\varphi\in \plainC{}(\R)\cap \plainL\infty(\R)$ be a function. 
Then 
\begin{align}\label{asb:eq}
\sum_{k=1}^N\underset{t\notin J}
\int \varphi(t)&\ \underset{|t-s|\ge T, s\in J_k}
\int \frac{1}{|t-s|^2} ds dt\notag\\[0.2cm]
 =  &\ \log\frac{1}{T}\sum_{k=1}^N
 \bigl(\varphi(s_k) 
 + \varphi(t_k)\bigr) + N \|\varphi\|_{\plainL\infty}O(1), \ T\to 0,
\end{align}
where $O(1)$ depends only on $d_0$ and $d_1$.
\end{prop}   
   
Now we can derive the following property of $\CS_j(\overc\bxi; T)$.

\begin{lem}
Let us fix a 
vector $\be\in\mathbb S^{d-1}$ and the index $j$. 
Let $S=S(\Psi_j)$ and 
$M = M(\Psi_j)$ be as defined in 
\eqref{surface:eq} and \eqref{Mmeasure:eq} respectively. 
Then for all $\overc\bxi\notin M$ we have 
\begin{equation}\label{cs:eq}
\CS_j(\overc\bxi; T) 
= \log\frac{1}{T}\sum_{\bxi: (\overc\bxi, t)\in S } \phi_j(\bxi)\ 
+\  \#\bigl(\SX_l(\overc\bxi; \Psi_j)\bigr)\, O(1),\ T\to\ 0.
\end{equation}
Moreover, 
\begin{equation}\label{ck_as:eq}
\underset{T\to 0}\lim\ \frac{1}{\log\frac{1}{T}}\CK_j(T)   
= \frac{U(1, 0; f)}{4\pi^2} 
\int_S \phi_j(\bxi) |\bn_{\bxi}\cdot\be| \,dS_{\bxi},
\end{equation}

\end{lem} 
 
\begin{proof} 
Let $\rho>0$ be such that $D_j\subset \CC_\rho^{(d)}$. 
By \eqref{Mmeasure:eq} the set 
\[
\SX_l(\overc\bxi) = \SX_l(\overc\bxi; \Psi_j) 
= \{t\in (-2\rho, 2\rho):(\overc\bxi, t)\in S\}
\] 
is finite for each $\overc\bxi\notin M$. 
Now the asymptotics \eqref{cs:eq} follow from \eqref{asb:eq}. 
Furthermore, 
\eqref{log_bound_up:eq} implies that
\begin{equation*}
|\CS_j(\overc\bxi; T)| \lesssim  \#\bigl(\SX_l(\overc\bxi)\bigr)\, \log\frac{1}{T},\ 
\forall \overc\bxi\notin M.
\end{equation*}
Thus, by the Dominated Convergence Theorem, \eqref{cs:eq} leads to the formula
\begin{align*}
\frac{1}{\log \frac{1}{T}}
\underset{\CC_\rho^{(d-1)}}\int \CS_j(\overc\bxi; T) d\overc\bxi 
\to\ \underset{\CC_\rho^{(d-1)}}\int 
\sum_{\bxi: (\overc\bxi, t)\in S } \phi_j(\bxi) d\overc\bxi,\ T\to 0.
\end{align*}
According to \eqref{SX1:eq}, the right-hand side 
coincides with 
\begin{equation*}
\underset{S}\int \phi_j(\bxi) |\bn_{\bxi}\cdot\be| \,dS_{\bxi},  
\end{equation*}
and hence \eqref{ck_as:eq} holds.
\end{proof}
 
 \subsection{Proof of Theorem \ref{comparison:thm}} 
Now we put together the formula \eqref{ck_as:eq}, the bound 
\eqref{ktilde:eq}, use the defintion \eqref{ck_improved:eq} 
and the bound \eqref{a2_diff:eq}. This leads to the asymptotic formula 
\begin{align*}
\underset{T\to 0}\lim\ 
\frac{1}{\log\frac{1}{T}}\CA_d^{(2)}(a_{T}, \be; f; \t T) 
= \frac{U(f)}{4\pi^2}
\int_{\p\Om} |\bn_{\bxi}\cdot\be| dS_{\bxi}.
\end{align*}
In view of the bound \eqref{ca1_bound:eq} the same formula 
formula holds for the coefficient $\CA_d(a_{T}, \be; f)$. 
On the other hand, this coefficient satisfies the bound 
\eqref{ca_bound:eq} uniformly in $\be\in\mathbb S^{d-1}$. 
Therefore, by the Dominated Convergence Theorem, we get for the 
integral $\CB_d$(see \eqref{cbd:eq}) the asymptotics
\begin{align*}
\underset{T\to 0}\lim\ 
\frac{1}{\log\frac{1}{T}}\CB_d(a_{T}; f) 
= \frac{U(f)}{(2\pi)^{d+1}}\int_{\p\L} 
\int_{\p\Om} |\bn_{\bxi}\cdot\bn_{\bx}| dS_{\bxi} d S_\bx,
\end{align*}
which coincides with the claimed formula \eqref{comparison:eq}.
\qed

\bibliographystyle{amsplain}

\end{document}